\documentclass[12pt]{amsart}
\usepackage{amsthm}
\usepackage{amsfonts}
\usepackage{xr}
\usepackage{graphicx}
\usepackage{amsmath}
\usepackage{epsfig,subfigure}
\usepackage{color}
\usepackage{amssymb}

\def\Z{\mathbb{Z}}
\def\R{\mathbb{R}}
\def\N{\mathbb{N}}

\def\epsilon{\varepsilon}
\def\hat{\widehat}
\def\tilde{\widetilde}

\newcommand{\me}{\mathrm{e}}

\newcommand{\SE}{\setcounter{equation}{0} \section}
\newcommand{\be}{\begin{equation}}
\newcommand{\ee}{\end{equation}}
\newcommand{\baa}{\begin{array}}
\newcommand{\eaa}{\end{array}}
\newcommand{\ba}{\begin{eqnarray}}
\newcommand{\ea}{\end{eqnarray}}
\numberwithin{equation}{section}

\newtheorem{theo}{\bf Theorem}[section]
\newtheorem{lem}[theo]{\bf Lemma}
\newtheorem{pro}[theo]{\bf Proposition}

\newtheorem{rem}[theo]{\bf Remark}

\newcommand{\Rd}{\color{red}}

\begin{document}
\date{\today}
\title[Spreading in space-time periodic media, Part 1]{Spreading in space-time periodic media governed by\\ a monostable  equation with free boundaries,\\ Part 1:  Continuous initial functions} 
\thanks{This work was supported by the Australian Research Council, the National Natural Science Fundation of China (11171319, 11371117) and the Fundamental Research Funds for the Central Universities. }
\author[w. Ding, Y. Du and X. Liang]{Weiwei Ding$^\dag$, Yihong Du$^\dag$ and Xing Liang$^\ddag$}
\thanks{
$^\dag$ School of Science and Technology, University of New England, Armidale, NSW 2351, Australia}
\thanks{$^\ddag$ School of Mathematical Sciences, University of Science and Technology of China, Hefei, Anhui, 230026, P.R. China}


\begin{abstract}
We aim to classify the long-time behavior of the solution to a free boundary problem with monostable reaction term in space-time periodic media. Such a model may be used to describe the spreading of a new or invasive species, with the free boundary representing the expanding front. 
In time-periodic and space homogeneous environment, as well as in space-periodic and time autonomous  environment,
such a problem has been studied recently in \cite{dgp, dl}. In both cases, a spreading-vanishing dichotomy has been established,
and when spreading happens, the asymptotic spreading speed is proved to exist by making use of the corresponding semi-wave solutions. The approaches in \cite{dgp, dl} seem difficult to apply to the current situation where the environment is periodic in both space and time. Here we take a different approach, based on the methods developed by Weinberger \cite{w1, w2} and others \cite{fyz,lyz,lz1,lz2,lui}, which yield the existence of the spreading speed without using traveling wave solutions.
 In Part 1 of this work, we establish the existence and uniqueness of classical solutions for the free boundary problem with continuous initial data, extending the existing theory which was established only for $C^2$ initial data. This will enable us to develop Weinberger's method in Part 2 to determine the spreading speed without knowing a priori the existence of the corresponding semi-wave solutions.  In Part 1 here, we also establish a spreading-vanishing dichotomy.
\end{abstract}

\maketitle


\SE{Introduction and main results}\label{sec1}
This work consists of two parts, and the current paper is Part 1.
The aim of this work is to classify the long-time dynamical behavior to a class of  space-time periodic reaction-diffusion equations with free boundaries of the form
\begin{equation}\label{eqf}\left\{\baa{ll}
u_t=du_{xx}+f(t,x,u),& g(t)<x<h(t),\quad t>0,\vspace{3pt}\\
u(t,g(t))=u(t,h(t))=0,&t>0,\vspace{3pt}\\
g'(t)=-\mu u_x(t,g(t)),&t>0, \vspace{3pt}\\
h'(t)=-\mu u_x(t,h(t)),&t>0, \vspace{3pt}\\
g(0)=g_0,\quad h(0)=h_0,\quad u(0,x)=u_0(x),& g_0\leq x\leq h_0,\eaa\right.
\end{equation} 
where $x=g(t)$ and $x=h(t)$ are the moving boundaries to be determined together with $u(t,x)$, and $\mu$ is a given positive constant. Throughout the paper, the diffusion coefficient $d$ is a positive constant; the reaction term $f:\R\times\R\times\R^+\mapsto\R$ is continuous, of class $C^{\alpha/2,\alpha}(\R\times\R)$ in $(t,x)\in\R\times\R$ locally uniformly in $u\in\R^+$(with $0<\alpha<1$), and of class $C^{1}$ in $u\in\R^+$ uniformly in $(t,x)\in\R\times\R$. The basic assumptions on $f$ are:
\begin{equation}\label{zero}
f(t,x,0)=0 \quad \hbox{for all  } t\in\R,\,\,x\in\R,
\end{equation}
there exists
$K>0$ such that 
\begin{equation}\label{globalb}
f(t,x,u)\leq Ku \quad \hbox{for all $u\geq 0$ and all $(t,x)\in\R^2$.}
\end{equation}

Later in the paper, we will  assume additionally that there is some  constant  $M>0$ such that 
\begin{equation}\label{hyp2}
f(t,x,u)\leq 0\,\hbox{ for all }\, t\in\R,\,x\in\R,\, u\geq M,
\end{equation}
and $f$ is  $\omega$-periodic in $t$ and $L$-periodic in $x$ for some positive constants $\omega$ and $L$, that is, 
\begin{equation}\label{period}
\left\{
\begin{array}{l}
f(t+\omega,x,u)=f(t,x,u)\\
 f(t,x+L,u)=f(t,x,u)
\end{array}\right. \hbox{ for all }\, (t,x)\in \R^2,\,u\geq 0.
\end{equation}
Let us note that since $f$ is $C^1$ in $u$,  \eqref{globalb} is satisfied whenever \eqref{zero} and \eqref{hyp2} hold.

The initial function $u_0$ belongs to $\mathcal{H}(g_0,h_0)$ for some $g_0<h_0$, where
\begin{equation*}
\mathcal{H}(g_0, h_0):=\Big\{\phi\in C([g_0,h_0]):\, \phi(g_0)=\phi(h_0)=0, \,\phi(x)>0 \hbox{ in }(g_0,h_0) \Big\}.
\end{equation*}

Free boundary problems of the type \eqref{eqf} arise naturally in many applied areas, such as  melting of ice in contact with 
water and spreading of invasive species; see, for example, \cite{ch2,cr,dlin,ru}. In this work, we regard \eqref{eqf} as describing 
the spreading of a new or invasive species over a one-dimensional habitat, where  $u(t,x)$ represents the population density of 
the species at location $x$ and time $t$, the reaction term $f$ measures the growth rate, the free boundaries $x=g(t)$ and 
$x=h(t)$ stand for the edges of the expanding population range, namely the spreading fronts. The Stefan conditions $g'(t)=-\mu u_x(t,g(t))$ and $h'(t)=-\mu u_x(t,h(t))$ may be 
interpreted as saying that  the spreading front expands at a speed  proportional to the population gradient at the front;
a deduction of these conditions from ecological considerations can be found in \cite{BDK}. When 
$f(t,x,u)$ is periodic with respect to $x$ and $t$ as described in \eqref{period}, problem \eqref{eqf} represents spreading of the speces in a  heterogeneous environment that 
is periodic in both space and time. 

\subsection{Related existing results and motivation}
Before going further, let us discuss the motivation of this work by firstly recalling some related known results.  In the case where the function $f$ does not depend on $x$ and $t$, and is of logistic type, that is, 
$$f(u)=u(a-bu)\hbox{ for some positive constants } a \hbox{ and } b,$$ 
such a problem was first studied in \cite{dlin} for the spreading of a new or invasive species. It is proved that, when
$$u_0\in C^2([g_0,h_0]),\, u_0(g_0)=u_0(h_0)=0, \,u_0(x)>0 \hbox{ in }(g_0,h_0), $$
there exists a unique solution $(u,g,h)$ with $u(t,x)>0$, $g'(t)<0$ and $h'(t)>0$ for all $t>0$ and $g(t)<x<h(t)$, 
and a spreading-vanishing dichotomy holds, namely, there is a barrier $R^*$ on the size of the population range, such that either 
\begin{itemize}
\item[(i)] {\bf Spreading}: the population range breaks the barrier at some finite time (i.e., $h(t_0)-g(t_0)\geq R^*$ for some $t_0>0$), and then the free boundaries go to infinity as $t\to\infty$ (i.e., $\lim_{t\to\infty}h(t)=\infty$ and $\lim_{t\to\infty}g(t)=-\infty$), and the population spreads to the entire space and stabilizes at its positive steady state (i.e. $\lim_{t\to\infty}u(t,x)=a/b $ locally uniformly in $x\in\R$) or   

\item[(ii)] {\bf Vanishing}: the population range never breaks the barrier (i.e. $h(t)-g(t)< R^*$ for all $t>0$), 
and the population vanishes (i.e. $\lim_{t\to\infty}u(t,x)=0$). 
\end{itemize}

Moreover, when spreading occurs, the asymptotic spreading speed can be determined, i.e., 
$$\lim_{t\to\infty}-g(t)/t=\lim_{t\to\infty}h(t)/t=c,
$$
 where $c$ is the unique positive constant such that  the problem
\begin{equation*}
\left\{\baa{l}
dq_{xx}-cq_x+q(a-bq)=0, \;q(x)>0 \quad\hbox{for } x\in (0,\infty),\vspace{3pt}\\
q(0)=0, \quad \mu q_x(0)=c,\quad q(\infty)=a/b
\eaa\right.
\end{equation*} 
has a (unique) solution $q$. Such a solution $q(x)$ is called a semi-wave with speed $c$.

These results have subsequently been extended to more general situations in several directions. Below, we only mention a few that are closely related to this work.

In the case where $f$ is $\omega$-periodic in $t$, radially symmetric in $x$, of logistic type and converges to some time periodic function $\bar{f}(t,u)$ as $|x|\to\infty$ with 
$$\bar{f}(t,u)=u(a_1(t)-b_1(t)u)\hbox{ for some positive $\omega$-periodic functions } a_1(t) \hbox{ and } b_2(t),$$ 
the existence of spreading speed is proved in \cite{dgp} by showing the existence and uniqueness of a positive time periodic semi-wave (see \cite[Theorem 2.5]{dgp}). When $f$ is radially symmetric in $x$, independent of $t$, of logistic type and converges to some  function $\hat{f}(|x|,u)$ as $|x|\to\infty$ with 
$$\hat{f}(r,u)=u(a_2(r)-b_2(r)u)\hbox{ for some positive  $L$-periodic functions } a_2(r) \hbox{ and } b_2(r),$$ 
the spreading speed is determined by the speed of the corresponding pulsating semi-wave (see \cite{dl}). In both cases, the existence of semi-waves is proved by a fixed point approach. 
Moreover, in the space-periodic case, a different method was used in \cite{zh} to prove the existence of pulsating semi-wave  (and hence the existence of spreading speed) for problem \eqref{eqf}, which is based on the approach developed in \cite{dgm}. 

In the recent work \cite{lls1,lls2},  the existence of time almost periodic semi-wave and spreading speed for problem \eqref{eqf} with time almost periodic monostable nonlinearity $f(t,u)$ are established.

\smallskip

When the function $f(t,x,u)$  varies with both the variables $t$ and $x$,  it seems difficult to adapt the approaches mentioned above to determine the spreading speed, mainly due to the difficulty to prove the existence of the corresponding semi-wave solutions.
The main goal of this work is to establish a different approach to treat  the space-time periodic case of problem \eqref{eqf}. 
We will focus on the monostable case and prove
a spreading-vanishing dichotomy, and then show the existence of spreading speed  when spreading happens. 

Our approach is based on developments of Weinberger's ideas firstly appeared in \cite{w1,w2}, where the existence of spreading 
speed for the corresponding Cauchy problem is proved without knowing the existence of the corresponding traveling wave 
solutions. However, to adapt these ideas to treat our free boundary problem here, it is necessary to firstly extend the existence 
and uniqueness theory for \eqref{eqf} with $C^2$ initial functions (see \cite{dlin}) to the case that the initial functions are 
merely continuous, which has not been considered before and requires  new techniques. 

Due to the different techniques used, and its length,  this work is divided into two separate papers.
The current paper constitutes Part 1, where we establish the existence and uniqueness theory for \eqref{eqf} with continuous initial functions, and also prove a spreading-vanishing dichotomy for \eqref{eqf}.  We will prove the existence of asymptotic spreading speed in Part 2 (see \cite{ddl}), based on the results obtained  here, and 
on Weinberger's ideas already mentioned above.



In the next two subsections, we describe the main results of this paper.

\subsection{Existence and uniqueness with continuous initial functions} For any $T>0$, by a {\it classical solution} of problem \eqref{eqf} for $0<t<T$ with initial function $u_0\in \mathcal{H}(g_0,h_0)$, we mean a triple  $\big(u(t,x),g(t),h(t)\big)$ such that $u\in C^{1,2}(G_T)\cap C(\overline{G_T})$, $g,\,h\in C^1((0,T])\cap C([0,T])$, and that all the identities in \eqref{eqf} are satisfied pointwisely in $G_T$, where
$G_T:=\big\{(t,x)\in\R^2:\,t\in(0,T],\,x\in [g(t),h(t)] \big\}.$

We note that the result below is for a rather general class of functions $f$, covering monostable, bistable and combustion types of nonlinearities, with no peridicity condition assumed.

\begin{theo}\label{existence}
Suppose that \eqref{zero} and \eqref{globalb} are satisfied. For any given $u_0\in\mathcal{H}(g_0,h_0)$, problem \eqref{eqf} admits a unique classical solution $\big(u(t,x),g(t),h(t)\big)$ defined for all $t>0$.
Moreover, for any $T>\tau>0$, 
\begin{equation}\label{loces}
\big\|u\big\|_{C^{1+\alpha/2,2+\alpha}(G_T^{\tau})}+\big\|g\big\|_{C^{1+\alpha/2}([\tau,T])} +\big\|h\big\|_{C^{1+\alpha/2}([\tau,T])} \leq C,
\end{equation}
\begin{equation}\label{g-h-es}
h_0\leq h(t)\leq h_0+Ht^{1/2}, \quad \quad  g_0-Ht^{1/2}\leq g(t)\leq g_0 \,\,\hbox{ for }\,\,  0\leq t\leq T,
\end{equation}
where $G_T^\tau=\big\{(t,x)\in\R^2:\,t\in[\tau,T],\,x\in[g(t),h(t)] \big\}$, $C$ and $H$ are  positive constants depending on $\tau$, $T$, $h_0-g_0$, $f$ and $\|u_0\|_{C([g_0,h_0])}$, with $H$ independent of $\tau\in (0, T)$.
\end{theo}

 By slight modifications of the proof and statements of  Theorem~\ref{existence}, this result can be extended to the case that  the initial function is bounded and piecewise continuous. Similar problems have been addressed for one-dimensional free boundary problems for the heat equation with  bounded piecewise continuous initial and boundary data in \cite{ch,chr,fpk}.

\subsection{Spreading-vanishing dichotomy} We now focus on monostable $f(t,x,u)$ that is periodic in both $t$ and $x$.
More precisely, we assume that the function $f$  satisfies  \eqref{zero}, \eqref{hyp2}, 
\eqref{period} and 
\begin{equation}\label{hyp1}
\forall\, (t, x) \in\R^2,\hbox{ the function }\, u\mapsto f(t,x,u)/u \,\hbox{ is decreasing for }\, u>0.
\end{equation}

We will show that whether spreading or vanishing happens partly depends on the sign of the generalized principal eigenvalue of the linear operator $\mathcal{L}$ defined by  
$$\mathcal{L}\psi:= \partial_t\psi-d\partial_{xx}\psi-\partial_uf(t,x,0)\psi \,\hbox{ for } \psi\in C_\omega^{1,2}(\R^2),$$
where
\[
C_\omega^{1,2}(\R^2):=\{ \phi\in C^{1,2}(\R^2),\; \phi(t+\omega, x)=\phi(t, x) \mbox{ for all } (t,x)\in\R^2\}.
\]
 The generalized principal eigenvalue of $\mathcal{L}$ is given by
\begin{equation}\label{princi}\left.\baa{ll}
\lambda_1(\mathcal{L})=\sup\big\{\lambda\in\R:& \hbox{there exists } \psi\in C_\omega^{1,2}(\R^2)  
 \vspace{3pt}\\
&\hbox{ such that } \psi>0 \hbox{ and } (\mathcal{L}-\lambda)\psi\geq 0 \hbox{ in } \R^2 \big\}.\eaa\right.
\end{equation}
In what follows, we assume that 
\begin{equation}\label{monostable}
\lambda_1(\mathcal{L})<0.
\end{equation}

An example of $f$ satisfying all these assumptions is the logistic nonlinearity 
\begin{equation}\label{logic}
f(t,x,u)=u\big(a(t,x)-b(t,x)u\big)
\end{equation}
where $a,\,b$ are of class $C^{\alpha/2,\alpha}$ which are $\omega$-periodic in $t$ and $L$-periodic in $x$, and there are positive constants $\kappa_1$, $\kappa_2$ such that 
$\kappa_1\leq a(t,x)\leq \kappa_2$ and $\kappa_1\leq b(t,x)\leq \kappa_2$ for all $(t,x)\in\R^2$.  These conditions may also be satisfied
with $a(t,x)$ sign-changing (see \cite{wmx}).

It is well known (see \cite{na1,na3}) that, under the above assumptions on $f$, the following problem 
\begin{equation}\label{psteady11}
\left\{\baa{l}
p_t=dp_{xx}+f(t,x,p)\,\, \hbox{ in }\,(t,x)\in\R^2,\vspace{3pt}\\
p(t,x)\,\hbox{ is $\omega$-periodic in $t$ and $L$-periodic in $x$},\eaa\right.
\end{equation}   
admits a unique positive solution $p(t,x)\in C^{1,2}(\R^2)$, and $p(t,x)$ is globally asymptotically stable in the sense that for any nonnegative bounded non-null initial function $v_0\in C(\R)$, there holds 
\begin{equation}\label{stablegene}
v(t+s,x;v_0) - p(t+s,x)\to 0\,\hbox{ as }\, s\to\infty  \, \hbox{ locally uniformly in }\, (t,x)\in\R^2,
\end{equation}
where $v(t,x;v_0)$ is the unique solution of the corresponding Cauchy problem 
\begin{equation}\label{cauchy}
\left\{\baa{ll}
v_t=dv_{xx}+f(t,x,v),& x\in\R,\,t>0,\vspace{3pt}\\
v(0,x)=v_0(x),& \,x\in\R.\eaa\right.
\end{equation}   

Before stating the spreading-vanishing dichotomy for problem \eqref{eqf}, let us introduce one more notation. Let $\big(u(t,x),g(t),h(t)\big)$ be the global classical solution of \eqref{eqf} with initial function $u_0\in \mathcal{H}(g_0,h_0)$. By the parabolic maximum principle and the Hopf lemma we easily deduce from the Stefan condition that $h'(t)>0$ and $g'(t)<0$ for all $t>0$. Therefore, the limits $\lim_{t\to\infty}h(t)$ and $\lim_{t\to\infty}g(t)$ exist and we denote them by $h_{\infty}$ and $g_{\infty}$, respectively.

\begin{theo}\label{spva1}
Suppose that \eqref{zero},  \eqref{hyp2}, \eqref{period}, \eqref{hyp1} and \eqref{monostable} are all satisfied. Then the following alternative hold:
Either 
\begin{itemize}
\item[{\bf (i)}] spreading happens, that is, $(g_{\infty}, h_{\infty})=\R$, and 
$$\lim_{t\to\infty}\big|u(t,x)-p(t,x) \big|=0 \,\,\hbox{ locally uniformly in }\, x\in\R,$$
where $p(t,x)$ is the unique positive solution of problem \eqref{psteady11}; 
or 
\item[{\bf (ii)}] vanishing happens, that is, there exists some constant $\overline{R}>0$ such that $(g_{\infty}, h_{\infty})$ is a finite interval with length no larger than $2\overline{R}$, and there holds
$$\lim_{t\to\infty}\max_{g(t)\leq x\leq h(t)}u(t,x)=0. $$
\end{itemize}
\end{theo}

(The positive constant $\overline{R}$ here can be determined; see \eqref{deoverr}). 

For any given initial function $u_0\in \mathcal{H}(g_0,h_0)$, we have the following criteria for spreading or vanishing. 

\begin{theo}\label{spva2}
Suppose that \eqref{zero},  \eqref{hyp2}, \eqref{period}, \eqref{hyp1} and \eqref{monostable} are all satisfied. Then there exists a positive constant $R^*$ such that 
\begin{itemize}
\item[(i)] if $(h_0-g_0)/2\geq  R^*$, then spreading always occurs;

\item[(ii)] if $(h_0-g_0)/2< R^*$, then there exists a unique $\mu^*>0$ depending on $u_0$ such that vanishing occurs if $0<\mu\leq  \mu^*$ and spreading occurs if $\mu>\mu^*$.
\end{itemize}
\end{theo}

In spatially periodic media, the critical size $R^*$ depends continuously and periodically on the value of $(g_0+h_0)/2$ (see \eqref{lambda*} and Lemma~\ref{proprin}), while in the spatially homogeneous case, $R^*$ is independent of $(g_0+h_0)/2$
 (see \cite{dlin,dlou}).

\subsection{Outline of the paper} The remaining part of this paper is organized as follows. Section~\ref{sec2} is divided into 3 subsections. 
In subsection 2.1, we give the proof of Theorem~\ref{existence}. In subsection 2.2, 
we prove the continuous dependence of the classical solutions on the intial data and some comparison results. In subsection 2.3,  we 
list without proof the corresponding results for a closely related problem of \eqref{eqf} (see \eqref{eqfunbd}), which will be used 
in Part 2 to determine the spreading speed. Section~\ref{sec3} is devoted to the proof of Theorems~\ref{spva1}~and~\ref{spva2}. 

\section{Existence, uniqueness and continuous dependence}\label{sec2}
This section is devoted to the proof of the existence and uniqueness of classical solutions for problem~\eqref{eqf}~as well as some
basic properties of these solutions. Throughout this section, we  assume that  $f$ satisfies  \eqref{zero} and \eqref{globalb}.

\subsection{Existence and uniqueness of classical solutions}\label{sec21}
For a given $u_0\in \mathcal{H}(g_0,h_0)$, we first prove the local existence of a classical solution and  the estimates~\eqref{loces},
\eqref{g-h-es}. Once we know the existence of a classical solution $(u, g,h)$ defined for $t\in (0, T]$ with some small $T>0$, then since $u(T,x)$ is a $C^2$ function one can apply the existing theory (see \cite{dlin}) to extend the solution to all $t>T$.

We prove the local existence result through an approximation argument.
Let $\epsilon_0=(h_0-g_0)/4$. For any given $u_0\in \mathcal{H}(g_0,h_0)$, we choose a nondecreasing sequence $\{u_{0n}\}_{n\in\N}\subset  C^2([g_0,h_0])$ such that for each  $n\in\N$, 
$$u_{0n}(x)=0\, \hbox{ for }\, x\in [g_0,g_{0n}]\cup [h_{0n},h_0], \quad
 0<u_{0n}(x)\leq u_{0}(x)\, \hbox{ for }\,x\in (g_{0n},h_{0n}), $$ 
where $g_{0n}=g_0+\epsilon_0/n$ and $h_{0n}=h_0-\epsilon_0/n$, and that 
$$u_{0n} \to u_0 \,\hbox{ in }\, C([g_0,h_0])\,\hbox{ as }\, n\to\infty. $$
It follows from \cite[Theorem 5.1]{dlin}\footnote{We remark that, although \cite[Theorem 5.1]{dlin} only deals  with problem~\eqref{eqf} with a special homogeneous logistic nonlinearity $f(t,x,u)=u(a-bu)$, its proof extends straightforwardly  to~\eqref{eqf} with a general  nonlinearity $f(t,x,u)$ satisfying~\eqref{zero}~and~\eqref{globalb}.} that for each $n\in\N$, problem~\eqref{eqf} admits a unique classical solution $(u_n,g_n,h_n)$ defined for all $t>0$ with 
\[
(u_{n}(0,x), g_n(0), h_n(0))=(u_{0n}(x), g_{0n}, h_{0n}) \mbox{ for $x\in[g_{0n},h_{0n}]$.}
\]
  Moreover, by the comparison principle for problem \eqref{eqf} with smooth initial values (see, e.g., \cite[Lemma~5.7]{dlin}), one obtains that for each $n\in\N$,
\begin{equation*}
g_{n+1}(t)\leq g_n(t),\,\,\,h_{n+1}(t)\geq h_n(t)\,\hbox{ for all }\,t>0,
\end{equation*}
and 
\begin{equation*}
0< u_n(t,x)  \leq u_{n+1}(t,x)\,\hbox{ for  } \, g_n(t)<x < h_n(t),\,\,\,t>0.
 \end{equation*}
On the other hand, it follows from the comparison principle again that 
\[
\mbox{$u_n\leq \tilde u$, $g_n\geq \tilde g$
and $h_n\leq \tilde h$,}
\]
 where $(\tilde u, \tilde g,\tilde h)$ is the classical solution to problem~\eqref{eqf}~with initial function $\tilde{u}_0\in C^2([g_0-1,h_0+1])$ such that $\tilde{u}_0> 0$ in $(g_0-1,h_0+1)$, $\tilde u_0(g_0-1)=\tilde u_0(h_0+1)=0$ and $\tilde{u}_0\geq u_0$ in $[g_0,h_0]$. As a consequence, there is a triple  $(u,g,h)$ such that 
\begin{equation}\label{mconvergh}
g(t)=\lim_{n\to\infty}g_n(t),\,\;\,h(t)=\lim_{n\to\infty}h_n(t)\, \hbox{ pointwisely for }\, t\geq 0, 
\end{equation}
and that 
\begin{equation}\label{mconveru}
u(t,x)=\lim_{n\to\infty}u_n(t,x)\,\hbox{ pointwisely for }\, g(t)< x< h(t),\, t\geq 0. 
\end{equation}

In what follows, we will show, via a sequence of lemmas, 
 that  $(u,g,h)$ is  a classic solution for problem~\eqref{eqf}~with initial function~$u_0$.

\begin{lem}\label{estderi}
Let $\big(u_n,g_n,h_n\big)$ be as  above. Then for any given $0<\tau_0<T_0$, there are positive constants $C_1$,
 $C_2$  independent of $n$ such that 
\begin{equation*}
0< u_n(t,x)\leq C_1 \,\,\hbox{ for }\,\, g_n(t)< x< h_n(t),\,0\leq t\leq T_0,
\end{equation*}
and 
\begin{equation*}
-C_2 \leq g_n'(t)<0, \,\,\, 0<h_n'(t)\leq C_2 \,\, \hbox{ for }\, \tau_0\leq t\leq T_0.
\end{equation*}
\end{lem}

\begin{proof}
Applying the parabolic maximum principle and the Hopf lemma to the equation of $u_n$, one immediately obtains that, for every $t>0$, 
$$u_n(t,x)>0 \hbox{ for }  g_n(t)<x<h_n(t),\,\,\,  \partial_x u_n\big(t,h_n(t)\big)<0\, \hbox{ and }\,\partial_x u_n\big(t,g_n(t)\big)>0. $$ 
It follows that $g_n'(t)<0 $ and $h_n'(t)<0$ for all $t>0$.

To find the bound $C_1$ for $u_n$, we make use of  \eqref{globalb}, and the comparison principle to obtain 
$$ u_n(t,x)\leq  \bar{u}_n(t)\,\hbox{ for }\, g_n(t)< x < h_n(t),\,0\leq t\leq T_0,$$
where $\bar{u}_n(t)$ solves
$$d \bar{u}_n /dt=K \bar{u}_n \mbox{ for } t>0;\quad \bar{u}_n(0)=\|u_{0n}\|_{C([g_0,h_0])}.  $$
Since $0\leq u_{0n}(x) \leq u_0(x) $ in  $[g_0,h_0]$ for all $n\in\N$, one can choose $C_1=\|u_{0} \|_{C([g_0,h_0])}\me^{KT_0}$, which clearly is independent of $n$.

We next show that 
\[
\mbox{$-C_2 \leq g_n'(t)$ and $h_n'(t)\leq C_2 $ for $\tau_0\leq t\leq T_0$}
\]
 with some positive constant $C_2$ which is independent of $n\in\N$. We only prove the estimate for $h_n'(t) $, since the estimate for $g_n'(t)$ can be proved analogously.

We first claim that, for any given $\tau_0$, there exists $n_0\in\N$ such that 
\begin{equation}\label{claimh}
h_{n}(\tau_0)>h_0\,\hbox{ for all }\,n\geq n_0.
\end{equation}

With $C_1$ determined above, since $f(t,x,0)=0$ and $f$ is $C^1$ in $u\in\R^+$, there exists $K_0>0$ such that
\[
f(t,x,u)\geq -K_0 u \mbox{ for } u\in [0, C_1],\; (t,x)\in\R^2.
\]
It follows that
\[
(u_n)_t-d(u_n)_{xx}\geq -K_0 u_n\; \mbox{ for } g_n(t)<x<h_n(t),\, 0\leq t\leq T_0.
\]
Hence $v_n(t,x):=e^{K_0t}u_n(t,x)$ satisfies
\[
(v_n)_t-d(v_n)_{xx}\geq 0 \mbox{ for } g_n(t)<x<h_n(t),\; 0<t\leq T_0,
\]
and
\[
g_n'(t)=-\mu (u_n)_x(t, g_n(t))\leq -\mu e^{-K_0T_0}(v_n)_x(t, g_n(t)) \mbox{ for } 0<t\leq T_0,
\]
\[
h_n'(t)=-\mu (u_n)_x(t, h_n(t))\geq -\mu e^{-K_0T_0}(v_n)_x(t, h_n(t)) \mbox{ for } 0<t\leq T_0.
\]

Since $u_n(t,x)\geq u_1(t,x)>0$ for $t\geq 0$ and $x\in (g_{01}, h_{01})$, there exists $\delta>0$ small such that
\[
u_n(t,x)\geq u_1(t,x)\geq \delta \mbox{ for } x\in [x_0-\delta, x_0+\delta]\subset (g_{01}, h_{01}),\; t\in [0, T_0],
\]
where $x_0:=(h_{01}-g_{01})/2$.

We now consider the auxiliary problem
\begin{equation}
\label{heat}
\left\{
\begin{array}{l}
w_t-dw_{xx}=0 \mbox{ for } x\in (x_0, s(t)), \; t\in (0, T_0],\\
w(t, x_0)=\delta,\; w(t, s(t))=0 \mbox{ for } t\in (0, T_0],\\
s'(t)=-\mu e^{-K_0T_0} w_x(t,s(t)) \mbox{ for } t\in (0, T_0],\\
w(0,x)=\delta \chi_{[x_0, x_0+\delta]}(x) \mbox{ for } x\in [x_0, h_0],\; s(0)=h_0.
\end{array}
\right.
\end{equation}
By \cite{chr}, \eqref{heat} has a classical solution $(w(t,x), s(t))$ and $s'(t)>0$ for $t\in (0, T_0]$. In particular, $s(\tau_0)>h_0$.

We next choose $n_0$ a large integer such that 
\[
\mbox{$h_{0n}>h_0-\min\{\delta, s(\tau_0)-h_0\}$ for $n\geq n_0$,}
\]
 and then define
\[ 
s_n(t)=s(t)-h_0+h_{0n} \mbox{ for } t\in [0, T_0],
\]
\[
w_n(t,x)=w(t, x-h_{0n}+h_0) \mbox{ for } x\in [x_0-h_0+h_{0n}, s_n(t)],\; t\in [0, T_0].
\]
By our choice of $n_0$ we have $x_{0n}:=x_0-h_0+h_{0n}\in [x_0-\delta, x_0]$ for $n\geq n_0$, and thus
\[
v_n(t,x)\geq u_n(t, x)\geq \delta \mbox{ for } t\in [0, T_0],\; x\in [x_{0n}, x_{0n}+\delta], \; n\geq n_0.
\]
Clearly $(w_n(t,x), s_n(t))$ satisfies
\[
\left\{
\begin{array}{l}
(w_n)_t-d(w_n)_{xx}=0 \mbox{ for } x\in (x_{0n}, s_n(t)), \; t\in (0, T_0],\\
w_n(t, x_{0n})=\delta,\; w_n(t, s_n(t))=0 \mbox{ for } t\in (0, T_0],\\
s_n'(t)=-\mu e^{-K_0T_0} (w_n)_x(t,s_n(t)) \mbox{ for } t\in (0, T_0],\\
w_n(0,x)=\delta \chi_{[x_{0n}, x_{0n}+\delta]}(x) \mbox{ for } x\in [x_{0n}, h_{0n}],\; s_n(0)=h_{0n}.
\end{array}
\right.
\]
Since $(v_n, h_n)$ is a super solution of the above problem, by the comparison principle, we obtain 
\[
h_n(t)\geq s_n(t)=s(t)-h_0+h_{0n} \mbox{ for } t\in (0, T_0],\; n\geq n_0.
\]
In particular,
\[
h_n(\tau_0)\geq s(\tau_0)-h_0+h_{0n}>h_0 \mbox{ for } n\geq n_0,
\]
as we claimed.
This proves \eqref{claimh}.

Next, set  $\delta_0=h_{n_0}(\tau_0)-h_0$ and consider the auxiliary problem  
\begin{equation}\label{auxielli}
dW_{xx}+\bar{f}(W)=0 \,\,\hbox{ for }\,\, -\delta_0<x<0, \quad
W(-\delta_0)=C \,\,\hbox{ and }\,\, W(0)=0,
\end{equation}  
where  $C=1+\max\{C_1, M\}$ with $M$ being the positive constant in the assumption \eqref{hyp2}, and $\bar{f}(s)$ is a function of class $C^1(\R^+)$ such that
$$\bar{f}(0)=\bar{f}(C)=0\,\, \hbox{ and }\,\, \bar{f}(s)\geq \sigma(s)f(t,x,s)\, \hbox{ for all }\, t\in\R,\,x\in\R, \,s\in [0,C],$$ 
where $\sigma(s)$ is a $C^1$  nonnegative function satisfying
\[
\sigma(s)=1 \mbox{ for } s\leq C_1,\; \sigma(C)=0.
\]
It is easy to see by a sub- and super-solution argument that problem \eqref{auxielli} admits a solution $W\in C^2([-\delta_0,0])$ such that 
$0< W(x)\leq C$ for all $-\delta_0\leq x< 0$. We now show that, for each given $t\in [\tau_0,T_0]$ and $n\geq n_0$,
\begin{equation}\label{comapp}
u_n(t,x)\leq W\big(x-h_n(t)\big) \, \hbox{ for all }\,  h_n(t)-\delta_0<x<h_n(t).
\end{equation}

For $n\geq n_0$ and fixed $t\in [\tau_0, T_0]$, since
\[
h_n(t)>h_n(t)-\delta_0\geq h_n(\tau_0)-\delta_0\geq h_{n_0}(\tau_0)-\delta_0=h_0> h_n(0),
\]
due to the monotonicity of $h_n(\tau)$ in $\tau$, there exists a unique $t_n\in (0, t)$ such that
$h_n(t_{n})=h_n(t)-\delta_0$.
We now apply the parabolic maximum principle to compare $u_n$ and $W$ over the region 
$$\Omega_n=\Big\{ (\tau,x):\, t_{n}<\tau\leq t,\, h_n(t)-\delta_0\leq x \leq h_n(\tau)\Big\}.$$
More precisely, set $\phi(\tau,x)=u_n(\tau,x)-W\big(x-h_n(t)\big)$ for $(\tau,x)\in \Omega_n$. 
It is straightforward to check that $u_n\big(t_{n}, h_n(t)-\delta_0\big)=u_n\big(t_{n}, h_n(t_n)\big)=0$, that 
$$ \phi\big(\tau, h_n(t)-\delta_0\big)= u_n\big(\tau, h_n(t)-\delta_0\big)- W(-\delta_0)\leq C_1-C\leq 0\,\hbox{ for all }\, t_{n}<\tau\leq t,$$
and that 
$$\phi\big(\tau,h_n(\tau)\big)= u_n\big(\tau, h_n(\tau)\big)- W\big(h_n(\tau)-h_n(t)\big) \leq 0\,\hbox{ for all }\, t_{n}<\tau\leq t. $$
On the other hand, by the assumptions on $\bar{f}$, it follows that there exists some  bounded function $b$ such that
\begin{align*}
\phi_{\tau}-d\phi_{xx} &= \sigma(u_n)f(\tau,x,u_n)-\bar{f}(W)\\
&\leq \sigma(u_n) f(\tau,x, u_n)-\sigma(W)f(\tau,x, W)\\
&=b(\tau,x)\phi\;\; \mbox{ for $(\tau,x)\in \Omega_n$.}
\end{align*}
One thus concludes from the parabolic maximum principle that $u_n(\tau, x)\leq W\big(x-h_n(t)\big)$ for any $(\tau,x)\in \Omega_n$.
This in particular implies the inequality \eqref{comapp} by choosing $\tau=t$.

To complete the proof, notice that $u_n\big(t,h_n(t)\big)= W(0)=0$. It then follows from \eqref{comapp} that $\partial_x u_n\big(t, h_n(t)\big)\geq W'(0)$ for all $n\geq n_0$, whence $-\mu^{-1} h'_n(t)\geq W'(0)$. This implies that $h'_n(t) \leq -\mu W'(0)$ for all $\tau_0\leq t\leq T_0$ and $n\geq n_0$.  By setting 
$$C_2=\max\big\{-\mu W'(0),\max_{0\leq n\leq n_0,\, \tau_0\leq t\leq T_0} h'_n(t)\big\},$$
one thus gets that $h'_n(t) \leq C_2$ for all $\tau_0\leq t\leq T_0$, $n\in\N$, and that $C_2$ only depends on $T_0$, $\tau_0$, $f$ and $\|u_0\|_{C([g_0,h_0])}$. The proof of Lemma~\ref{estderi} is thereby complete.
\end{proof}

\begin{lem}\label{est-g_n-h_n}
Let $g_n$ and $h_n$ be as in Lemma \ref{estderi}. Then for any given $T_0>0$,
 there exists some positive constant $H$ independent of $n$ such that 
\begin{equation}\label{boundfhn}
h_{0n}\leq h_n(t)\leq h_{0n}+Ht^{1/2}\quad\hbox{and}\quad  g_{0n}-Ht^{1/2}\leq g_n(t)\leq g_{0n} \,\,\,\hbox{ for all } \,\, 0\leq t\leq T_0.
\end{equation}
\end{lem}
\begin{proof}
For any given $T_0>0$ and each $n\in\N$, consider the following free boundary problem 
\begin{equation}\label{degenrate}\left\{\baa{l}
\partial_t v_n=d\partial_{xx}v_n, \quad h_{0n}<x<\tilde{h}_n(t),\,\,\,0<t\leq T_0,\vspace{3pt}\\
v_n(t,h_{0n})=\tilde{C}, \,\,\, v_n\big(t,\tilde{h}_n(t)\big)=0,\quad 0<t\leq T_0,\vspace{3pt}\\
 \tilde{h}_n'(t)=-\mu \me ^{KT_0}\partial_xv_n(t,\tilde{h}_n(t)),\quad 0<t\leq T_0, \vspace{3pt}\\
\tilde{h}_n(0)= h_{0n},\eaa\right.
\end{equation} 
where $K$ is the positive constant given in \eqref{globalb}, and $\tilde{C}$ is some positive constant to be chosen independent of $n$ later. It follows from \cite[Theorem 1]{ch} that problem \eqref{degenrate} admits a unique classical solution $(v_n, \tilde{h}_n)$ with $\tilde{h}_n\in  C^1((0,T_0])$ and  $\tilde{h}_n$ being H\"{o}lder continuous at $t=0$ with exponent $1/2$. Namely, there exists some positive constant $H$  such that 
\begin{equation}\label{auxihn}
h_{0n}\leq \tilde{h}_n(t)\leq h_{0n}+Ht^{1/2}\,\hbox{ for all }\, 0\leq t\leq T_0.
\end{equation}
Furthermore, for any $n_1\in\N$ and $n_2\in\N$,  it is straightforward to check that $\big(v_{n_1}(t,x-h_{0n_2}+h_{0n_1}), \tilde{h}_{n_1}(t)+h_{0n_2}-h_{0n_1}\big)$ is the solution of problem \eqref{degenrate} with $n=n_2$. Thus, by the uniqueness of such solutions, one concludes that $H$ is independent of $n$.

Next, for any fixed $n\in\N$, due to the assumption \eqref{globalb}, it is easy to see from the comparison principle for problem \eqref{eqf} with smooth initial values (see, e.g., \cite[Lemma~5.7]{dlin}) that
\begin{equation}\label{auxihn2}
\bar{g}_{n}(t)\leq g_n(t),\,\,\,\bar{h}_{n}(t)\geq h_n(t)\,\hbox{ for all }\,0<t\leq T_0,
\end{equation}
and 
\begin{equation*}
0< u_n(t,x)  \leq \me^{Kt}\bar{u}_{n}(t,x)\,\hbox{ for  all } \, g_n(t)<x < h_n(t),\,\,\,0<t\leq T_0,
 \end{equation*}
where $(\bar{u}_{n},\bar{g}_{n}, \bar{h}_{n})$ is the classical solution of the following free boundary problem
\begin{equation*}\left\{\baa{ll}
\partial_t\bar{u}_{n}=d\partial_{xx}\bar{u}_{n},& \bar{g}_{n}(t)<x<\bar{h}_{n}(t),\quad 0<t\leq T_0,\vspace{3pt}\\
\bar{u}_{n}(t,\bar{g}_{n}(t))=\bar{u}_{n}(t,\bar{h}_{n}(t))=0,& 0<t\leq T_0,\vspace{3pt}\\
\bar{g}_{n}'(t)=-\mu\me^{KT_0} \partial_x\bar{u}_{n}(t,\bar{g}_{n}(t)),&0<t\leq T_0, \vspace{3pt}\\
\bar{h}_{n}'(t)=-\mu \me^{KT_0}\partial_x\bar{u}_{n}(t,\bar{h}_{n}(t)),&0<t\leq T_0, \vspace{3pt}\\
\bar{g}_{n}(0)=g_{0n},\,\,\, \bar{h}_{n}(0)=h_{0n},\,\,\, u(0,x)=u_{0n}(x),& g_{0n}\leq x\leq h_{0n}.\eaa\right.
\end{equation*} 
Since $\bar{u}_{n}(t,x)$ is uniformly bounded for $\bar{g}_n(t)\leq x \leq  \bar{h}_n(t)$, $0\leq t\leq T_0$, one finds some $\tilde{C}>0$ such that 
$$\bar{u}_{n}(t,h_{0n})\leq \tilde{C}\,\hbox{ for all }\, 0\leq t\leq T_0,\, n\in\N. $$
It then follows directly from the comparison principle for problem \eqref{degenrate} established in \cite[Theorem~2]{ch} that 
$$\bar{h}_n(t)\leq \tilde{h}_n(t)\,\hbox{ for all }\,0\leq t\leq T_0,\,n\in\N.$$ 
This together with \eqref{auxihn} and \eqref{auxihn2} implies that $h_{0n}\leq h_n(t)\leq h_{0n}+Ht^{1/2}$ for all $0\leq t\leq T_0$ and $n\in\N$. In a similar way, one can prove the corresponding estimate for $g_n$ in \eqref{boundfhn}.
\end{proof}

Next, we prove that the limit $(u,g,h)$ given in \eqref{mconvergh} and \eqref{mconveru} is a classical solution for problem \eqref{eqf} over $G_T$ for some $T>0$. We prove this in the next two lemmas.

 \begin{lem}\label{locex}
Let  $\big(u(t,x),g(t),h(t)\big)$ be the limit given in \eqref{mconvergh} and \eqref{mconveru}. Then there is $T>0$ such that for $t\in (0, T]$, the first four equations in \eqref{eqf} are satisfied by  $\big(u(t,x),g(t),h(t)\big)$. 
\end{lem} 

\begin{proof}
We adopt the notations $u_n$, $g_n$, $h_n$, $T_0$, $C_1$ and $C_2$ used in Lemma~\ref{estderi}.
  We first straighten the free boundaries of problem~\eqref{eqf}~as in \cite{cf,dlin}. Without loss of generality, we assume that $g_{0}<0<h_{0}$. Then there is some $n_0\in\N$ such that for $n\geq n_0$,
there holds $g_{0n}<0<h_{0n}$, and there exist functions $\xi_+, \xi_-\in C^3(\R)$ satisfying 
$$\xi_+(y)=1 \,\hbox{ if }\,  |y-h_{0n}|< \frac{h_{0}}{4}, \,\,\,\xi_+(y)=0 \, \hbox{ if }\,  |y-h_{0n}|> \frac{h_{0}}{2},\,\,\, |\xi_+'(y)|<\frac{3}{h_{0}} \,\hbox{ for }\, y\in\R, $$
and 
$$\xi_-(y)=1 \,\hbox{ if }\,  |y-g_{0n}|< -\frac{g_{0}}{4}, \,\,\,\xi_-(y)=0 \, \hbox{ if }\,  |y-g_{0n}|> -\frac{g_{0}}{2},\,\,\, |\xi_-'(y)|<-\frac{3}{g_{0}} \,\hbox{ for }\, y\in\R. $$

For any fixed $n\geq n_0$, consider the transformation $(t,y)\to (t,x)$ given by 
$$ x=\phi_n(t,y):=y+\xi_+(y)(h_n(t)-h_{0n})+\xi_-(y)(g_n(t)-g_{0n}) \,\hbox{ for }\, 0\leq t\leq T_0,\, y\in\R. $$
Due to the inequalities in \eqref{boundfhn}, there is a positive constant $T\leq T_0$ (independent of $n$) small enough such that 
\begin{equation*}
| h_n(t)- h_{0n}  | \leq \frac{h_{0}}{8}\,\,\, \hbox{and}\,\,\,  | g_n(t)-g_{0n}  | \leq -\frac{g_{0}}{8}\, \hbox{ for all }\, t\in [0, T],\,n\geq n_0,  
\end{equation*}
whence the above transformation is a diffeomorphism from $[0,T]\times\R$ to $[0,T]\times\R$. Moreover, under this transformation, the free boundaries $x=h_n(t)$, $x=g_n(t)$ correspond to the straight lines $y=h_{0n}$ and  $y=g_{0n}$, respectively. 

Set 
$$w_n(t,y):=u_n\big(t, \phi_n(t,y)\big)$$
and
\begin{equation*}
\left.\begin{array}{rl}
\displaystyle {A_n(t,y)} :=& \displaystyle\frac{1}{1+\xi_+'(y)(h_n(t)- h_{0n})+\xi_-'(y)(g_n(t)-g_{0n})},\vspace{4pt}\\
\displaystyle B_n(t,y):=& \displaystyle\frac{\xi_+''(y)(h_n(t)- h_{0n})+\xi_-''(y)(g_n(t)-g_{0n})}{[1+\xi_+'(y)(h_n(t)- h_{0n})+\xi_-'(y)(g_n(t)-g_{0n})]^3},\vspace{4pt}\\
\displaystyle C_n(t,y):=&\displaystyle  \frac{h_n'(t)\xi_+(y)+g_n'(t)\xi_-(y)}{1+\xi_+'(y)(h_n(t)- h_{0n})+\xi_-'(y)(g_n(t)-g_{0n})}. \end{array}\right.
\end{equation*}
Then a simple calculation gives
\[
(u_n)_t=(w_n)_t-C_n(w_n)_y,\; (u_n)_x={A_n}(w_n)_y,
\]
 and 
\[
(u_n)_{xx}=A_n^2(w_n)_{yy}-B_n(w_n)_y,
\]
 whence $w_n$ satisfies
\begin{equation}\label{inibdeq}
\left\{\baa{ll}
(w_n)_t- dA_n^2(w_n)_{yy}+\big(dB_n-C_n\big)(w_n)_y \;\; &  \vspace{3pt}\\
\;\;\;\;\;\;\;\; \;\;\;\; \;\;\;\; =f\big(t,\phi_n(t,y), w_n\big),\; &(t,y)\in (0,T]\times(g_{0n}, h_{0n}),  \vspace{3pt}\\
w_n(t,h_{0n})=w_n(t,g_{0n})=0,\;\; & 0<t\leq T,\vspace{3pt}\\
w_n(0,y)=u_{0n}(y),\;\; & g_{0n}\leq y\leq h_{0n},\eaa\right.
\end{equation}
and $g_n,\,h_n$ satisfy, due to $A_n(t,y)=1$ for  $y\in\{g_{0n}, h_{0n}\}$,
\begin{equation}\label{inibdbd}\left\{\baa{l}
h_n'(t)=-\mu (w_n)_y(t,h_{0n}), \,\,\, 0<t\leq T,\vspace{3pt}\\
g_n'(t)=-\mu (w_n)_y(t,g_{0n}), \,\,\, 0<t\leq T,\vspace{3pt}\\
h_n(0)=h_{0n},\,\,\, g_n(0)=g_{0n}.\eaa\right.
\end{equation}

Next, we show some further estimates for $(w_n, g_n, h_n)$.
It follows from Lemma~\ref{estderi} that $w_n(t,y)$ is positive and uniformly bounded with respect to $n\in\N$ in $(t,y)\in [0,T]\times [g_{0n},h_{0n}]$. Moreover, the coefficients $A_n(t,y)$, $B_n(t,y)$ and $C_n(t,y)$ are all uniformly bounded and continuous in $(t,y)\in [\tau,T]\times (g_{0n},h_{0n})$ for any given $0<\tau<2\tau<T$. 
Then by applying parabolic $L^p$ theory (see, e.g., \cite[Theorem 7.15]{lieb}) and then Sobolev imbedding theorem, one obtains $w_n\in C^{(1+\alpha)/2,1+\alpha}([\tau,T]\times[g_{0n},h_{0n}])$, and 
\begin{equation*}
\big\| w_n \big\|_{C^{(1+\alpha)/2,1+\alpha}([\tau,T]\times[g_{0n},h_{0n}])} \leq C_3 \quad \hbox{for all }\, n\geq n_0,
\end{equation*}
where $C_3$ is a positive constant depending on $\tau$, $T$, $h_0-g_0$, $\| u_{0} \|_{C([g_{0},h_{0}])}$, $C_1$ and $C_2$ (which are given in Lemma~\ref{estderi}).  This together with~\eqref{inibdbd} implies that $g_n,\,h_n\in C^{1+\alpha/2}([\tau,T])$, and there exists $C_4>0$ independent of $n$ such that
\begin{equation*}
\big\| g_n \big\| _{C^{1+\alpha/2}([\tau,T])}\leq C_4,\quad \big\| h_n \big\| _{C^{1+\alpha/2}([\tau,T])}\leq C_4\quad \hbox{for all } \, n\geq n_0.
\end{equation*}
This implies that $\phi_n(t,y), A_n(t,y), B_n(t,y)$ and $C_n(t,y)$ are functions in $C^{\alpha/2, \alpha}([\tau, T]\times \R)$ and their norms in this space  have a bound independent of $n$.
We may now apply the parabolic Schauder estimates to problem~\eqref{inibdeq}, to obtain that $w_n\in C^{1+\alpha/2,2+\alpha}([2\tau,T]\times[g_{0n},h_{0n}])$, and 
\begin{equation*}
\big\| w_n \big\|_{C^{1+\alpha/2,2+\alpha}([2\tau,T]\times[g_{0n},h_{0n}])} \leq C_5 \quad \hbox{for all }\, n\geq n_0,
\end{equation*}
for some constant $C_5$ independent of $n$. 
Thus, one has
\begin{equation}\label{applocess}
\big\| u_n \big\|_{C^{1+\alpha/2,2+\alpha}(G_T^{2\tau})} + \big\| g_n \big\| _{C^{1+\alpha/2}([2\tau,T])}+ \big\| h_n \big\| _{C^{1+\alpha/2}([2\tau,T])}\leq C_6\quad \hbox{for any } \, n\geq n_0,
\end{equation}
for some positive constant $C_6$ independent of $n$, where  
$$G_{T,n}^{2\tau}=\big\{(t,x)\in\R^2:\,t\in[2\tau,T],\,x\in[g_n(t),h_n(t)] \big\}.$$

Finally, by using a diagonal argument and the convergences \eqref{mconvergh}, \eqref{mconveru}, one sees that $(u,g,h)\in C^{1,2}(G_T)\times C^1((0,T])\times C^1((0,T])$, and that
\begin{equation*}
u_n\to u\,\hbox{ in }\, C_{loc}^{1,2}(G_T) \, \hbox{ as }\, n\to\infty, 
\end{equation*}
\begin{equation*}
g_n\to g\, \hbox{ and }  h_n\to h \,\hbox{ in }\, C_{loc}^1((0,T])  \, \hbox{as }\, n\to\infty.
\end{equation*}
In particular, this implies that 
\begin{equation*}
u_t=du_{xx}+f(t,x,u)\,\, \hbox{ for all }\, g(t)<x<h(t),\,0<t<T.
\end{equation*}
Furthermore, for any $t\in [2\tau, T]$ and $x\in (g(t), f(t))$, there exists $n_1\in\N$ such that $x\in \big(g_n(t),h_n(t)\big)$ for all $n\geq n_1$, whence
\begin{equation*}
\big|u_n(t,x)\big|\leq C_6\big|x-g_n(t)\big|, \,\,\,\big|(u_n)_x(t,x)- (u_n)_x(t,g_n(t)) \big|\leq C_6\big|x-g_n(t)\big|\,\, \hbox{ for all }\,n\geq n_1,
\end{equation*}
where $C_6$ is the positive constant given in \eqref{applocess} (independent of $n$).
Passing to the limit  $n\to\infty$ in the first inequality gives that $\big|u(t,x)\big|\leq C_6\big|x-g(t)\big|$, which clearly implies $u(t,g(t))=0$. 
Similarly, due to $(u_n)_x(t, g_n(t))=-\mu^{-1} g_n'(t)$, passing to the limit  $n\to\infty$ followed by letting $x\to g(t)$ in the second inequality yields that $u_x(t,g(t))=-1/\mu g'(t)$. Since $\tau$ can be chosen arbitrarily in $(0,T/2]$, one thus obtains that
\begin{equation*}
u(t,g(t))=0\,\hbox{ and }\, g'(t)=-\mu u_x(t,g(t))\quad \hbox{for all }\,\, 0<t<T.
\end{equation*}
In a similar way, one concludes that
\begin{equation*}
u(t,h(t))=0\,\hbox{ and }\, h'(t)=-\mu u_x(t,h(t))\,\, \hbox{ for all }\, 0<t<T.
\end{equation*}
The proof is complete.
\end{proof}

\begin{lem}\label{ini-cond} The triple  $(u,g,h)$ in Lemma \ref{locex} also satisfies the initial conditions in 
 \eqref{eqf}. That is,
\begin{equation}\label{coninuousgh}
\lim_{t\to 0} g(t)=g_0,\quad \quad\lim_{t\to 0} h(t)=h_0,
\end{equation}
and for any $x_0\in[g_0,h_0]$,
\begin{equation}\label{continuousu}
\lim_{(t,x)\in G_T, t\to0,x\to x_0} u(t,x)=u_0(x_0).
\end{equation}
\end{lem}
\begin{proof}
Letting $n\to\infty$ in \eqref{boundfhn} we immediately obtain
\begin{equation}
\label{g-h-Holder}
h_0\leq h(t)\leq h_0+H t^{1/2},\; g_0-H t^{1/2}\leq g(t)\leq g_0 \mbox{ for } t\in (0, T_0].
\end{equation}
This clearly implies 
 \eqref{coninuousgh}. 

Next, we prove \eqref{continuousu}. Let $\bar{u}_0\in C([g(T),h(T)])$ be a nonnegative function such that $\bar{u}_0(x)= u_0(x)$ for $x\in [g_0,h_0]$ and 
$\bar{u}_0(x)=0$ for $x\in [g(T),g_0]\cup[h_0,h(T)]$. It follows from the parabolic comparison principle that 
\begin{equation*}
0\leq  u_n(t,x)  \leq \bar{u}(t,x)\,\hbox{ for  all } \, g_n(t)<x < h_n(t),\,\,\,0<t\leq T,\,n\in\N,
 \end{equation*}
where $\bar{u}(t,x)$ is the unique solution of the following initial-boundary value problem 
\begin{equation*}
\left\{\baa{ll}
\bar{u}_t=d\bar{u}_{xx}+K\bar{u},& g(T)<x<h(T),\quad 0<t<T,\vspace{3pt}\\
\bar{u}(t,g(T))=\bar{u}(t,h(T))=0,&0<t<T,\vspace{3pt}\\
\bar{u}(0,x)=\bar{u}_0(x),& g(T)\leq x\leq h(T),\eaa\right.
\end{equation*} 
with $K$ being the constant given in \eqref{globalb}.
This together with the convergence property \eqref{mconveru} implies that 
$$0< u(t,x)  \leq \bar{u}(t,x)\,\hbox{ for  all } \, g(t)<x < h(t),\,\,\,0<t\leq T.$$
Furthermore, since $\bar{u}_0\in C([g(T),h(T)])$, by the parabolic regularity theory on the boundary (see, e.g., \cite[Theorem 9 in Chapter 3]{fried2}), one has $\bar{u}\in C([0,T]\times[g(T),h(T)])$. 

For any $x_0\in (g_0,h_0)$ and any sequence $(t_m,x_m)_{m\in\N}\subset \R^2$ with $\lim_{m\to\infty}t_m=0$ and $\lim_{m\to\infty}x_m=x_0$,  there exists $n_2\in\N$ such that $g_n(t_m)<g_{0n}< x_m< h_{0n}< h_n(t_m)$ for all $n\geq n_2$, 
$m\geq n_2$, whence $ u_n(t_m,x_m)\leq u(t_m,x_m)\leq \bar{u}(t_m,x_m)$. This together with the facts that $\lim_{m\to\infty} u_n(t_m,x_m)=u_{0n}(x_0)$ for all $n\geq n_2$ and that $\lim_{m\to\infty}\bar{u}(t_m,x_m)=\bar{u}_0(x_0)$ implies that 
$$u_{0n}(x_0) \leq \liminf_{m\to\infty}u(t_m,x_m)\leq  \limsup_{m\to\infty}u(t_m,x_m)\leq \bar{u}_0(x_0).$$
Since $u_{0n}(x_0)$ converges to $u_0(x_0)$ uniformly in $x_0\in [g_0,h_0]$ as $n\to\infty$ and $\bar{u}_0(x_0)=u_0(x_0)$, it follows that $\lim_{m\to\infty}u(t_m,x_m)=u_0(x_0)$. Due to the arbitrariness of the sequence $(t_m,x_m)_{m\in\N}$, one obtains the property \eqref{continuousu} for all $x_0\in (g_0,h_0)$. 

In the case where $x_0=g_0$ or $x_0=h_0$, we have $\lim_{(t,x)\in G_T, t\to0,x\to x_0} \bar u(t,x)=\bar u_0(x_0)=0$.
Thus it follows from $0\leq u(t,x)\leq \bar u(t,x)$ in $G_T$ that
\[
\lim_{(t,x)\in G_T, t\to0,x\to x_0}  u(t,x)=0=u_0(x_0).
\]
Hence \eqref{continuousu} holds for all $x_0\in [g_0,h_0]$.
The proof of Lemma~\ref{ini-cond} is thereby complete.
\end{proof}

\begin{lem}\label{lem-existence}
For any $u_0\in\mathcal{H}(g_0,h_0)$, \eqref{eqf} has a classical solution defined for all $t>0$, and it satisfies \eqref{loces} and \eqref{g-h-es}.
\end{lem}
\begin{proof} We already obtained in the previous lemmas  a classical solution $(u,g,h)$ of \eqref{eqf} which is defined for $t\in (0, T]$ with $T>0$ sufficiently small.
Moreover, by \eqref{applocess}, it is easy to see that this solution satisfies \eqref{loces} for such $T$ and $\tau\in (0, T)$.
Thus $u(T/2, x)$ is a $C^2$ function meeting the requirement for the initial function in \cite{dlin}. It follows that this solution can be extended uniquely to all $t>T/2$ by the existence theory in \cite{dlin}, and it satisfies \eqref{loces} for $T>\tau>0$ with an arbitrary $T>0$.
Finally \eqref{g-h-es} follows from \eqref{g-h-Holder} and \eqref{applocess}.
\end{proof}

\begin{rem} \label{approx-above} Analogously, for any given $u_0\in \mathcal{H}(g_0,h_0)$,  we choose a decreasing sequence of intervals $[\tilde g_{0n}, \tilde h_{0n}]$ such that $\tilde g_{0n}\nearrow g_0$, $\tilde h_{0n}\searrow h_0$ as $n\to\infty$, and a sequence of functions $\tilde u_{0n}\in C^2([\tilde g_{0n}, \tilde h_{0n}])$ such that 
\[
\tilde u_{0n}(\tilde g_{0n})=\tilde u_{0n}(\tilde h_{0n})=0,\; \tilde u_{0n}>0 \mbox{ in }  (\tilde g_{0n}, \tilde h_{0n}),\; n\in\N,
\]
and that after extending $\tilde u_{0n}(x)$ and $u_0(x)$ to $\R$ by the value zero outside their supporting sets,
\[
 \tilde u_{0(n-1)}\geq \tilde u_{0n} \mbox{ in } \R, \; \lim_{n\to\infty}\|\tilde u_{0n}-u_0\|_{L^\infty(\R)}=0.
\]
Denoting by $(\tilde u_n, \tilde g_n, \tilde h_n)$ the unique solution of \eqref{eqf} with 
$(u_0, g_0, h_0)=(\tilde u_{0n}, \tilde g_{0n}, \tilde h_{0n})$, then we can similarly show that 
$(\tilde u_n, \tilde g_n, \tilde h_n)$ satisfies \eqref{boundfhn}, \eqref{applocess}, and   converges to a classical solution $(\tilde u, \tilde g, \tilde h)$ 
of \eqref{eqf} with initial data $(u_0, g_0, h_0)$ for $t\in (0, T]$ with $T>0$ small, which can be 
 extended to a classical solution of \eqref{eqf} for all $t>0$, and it satisfies \eqref{loces} and \eqref{g-h-es}.
\end{rem}

Now we proceed to prove the uniqueness of classical solutions to \eqref{eqf}. We will adapt the week solution approach in \cite{dg2} for higher space dimensions to the one space dimension setting here.

\begin{lem}\label{clawek}
Assume that $(u,g,h)$ is a classical solution for~\eqref{eqf}~defined over $G_T$ for some $T>0$ with initial function $u_0\in \mathcal{H}(g_0,h_0)$. For any given open interval $I$ such that $[g(T), h(T)]\subset I$, denote $I_T=(0,T]\times I$, and  
\begin{equation}\label{weakdu}
\tilde{u}(t,x)=\!\left\{\begin{array}{lll}
\!\! u(t,x) & \hbox{for }x\in [g(t),h(t)],\,\,\,0\leq t\leq T,\vspace{3pt}\\
\!\!0 & \hbox{for }x\in I\setminus [g(t),h(t)],\,\,\,0\leq t\leq T.\end{array}\right.
\end{equation}
Then $\tilde{u}\in C(\overline{I_T})$ and 
\begin{equation}\label{dweakf}
\int\limits_{0}^T\!\!\int\limits_{I} \big[d\tilde{u}\phi_{xx}+\kappa(\tilde{u})\phi_t\big]dxdt+\int\limits_{I}\kappa(\tilde{u}_0)\phi(0,x)dx+\int\limits_{0}^T\!\!\int\limits_{I} f(t,x,\tilde{u})\phi dxdt=0
\end{equation}
for every function $\phi \in C(\overline{I_T})\cap W^{1,2}(I_T)$ such that $\phi=0$ on $(\{T\}\times I)\cup([0,T]\times \partial I)$, where $\kappa(\cdot)$ is a function defined by $\kappa(w)=w$ if $w>0$ and $\kappa(w)=w-\mu^{-1}d$ if $w\leq 0$.
\end{lem}

\begin{proof}
By the definition of $\tilde{u}$, clearly $\tilde{u}\in C(\overline{I_T})$.
We now prove  that $\tilde{u}$ satisfies \eqref{dweakf}~for every  $\phi \in C(\overline{I_T})\cap W^{1,2}(I_T)$ such that $\phi=0$ on $(\{T\}\times I)\cup([0,T]\times \partial I)$. To do so, we multiply both sides of the first equation in~\eqref{eqf}~by $\phi$ and integrate over $G_T^{\tau}$ for any given $0<\tau<T$. Since $u(t,g(t))=u(t,h(t))=0$ for all $0<t<T$,  integration by parts yields
\begin{equation*}
 -\int\limits_{\tau}^T\int\limits_{g(t)}^{h(t)}\big[u\phi_t + du\phi_{xx} \big]dxdt\, - \int\limits_{g(\tau)}^{h(\tau)}u(\tau,x)\phi(\tau,x)dx\vspace{3pt}
=d\int\limits_{\tau}^TJ(t)dt+\int\limits_{\tau}^T\int\limits_{g(t)}^{h(t)}f(t,x,u)\phi dxdt.
\end{equation*}
where
$$J(t)= u_x(t,h(t))\phi(t,h(t))- u_x(t,g(t))\phi(t,g(t)).$$
By elementary calculus,
\begin{equation*}
\begin{split}
\int\limits_{\tau}^T\int\limits_{I\setminus [g(t),h(t)]} \phi_t dxdt 
\,=&\int\limits_{\tau}^T\big[\phi(t,h(t))h'(t)-\phi(t,g(t))g'(t)\big] dt-\int\limits_{I\setminus [g(\tau),h(\tau)]} \phi(\tau,x)dx\vspace{3pt}\\
=&-\mu\!\int\limits_{\tau}^TJ(t) dt - \int\limits_{I\setminus [g(\tau),h(\tau)]} \phi(\tau,x)dx.\vspace{3pt}\\
\end{split}
\end{equation*}
Combining the above, since $f(t,x,0)\equiv 0$, we obtain
$$\int\limits_{\tau}^T\!\!\int\limits_{I} \big[d\tilde{u}\phi_{xx}+\kappa(\tilde{u})\phi_t\big]dxdt+\int\limits_{I}\kappa(\tilde{u}(\tau,x))\phi(\tau,x)dx+\int\limits_{\tau}^T\!\!\int\limits_{I} f(t,x,\tilde{u})\phi dxdt=0.
$$
Since $\phi \in C(\overline{I_T})\cap W^{1,2}(I_T)$ and $\tilde{u}\in C(\overline{I_T})$ (and hence, $\kappa(\tilde{u})$ is bounded in $I_T$), passing to the limit as $\tau\to 0$ in the above equality gives \eqref{dweakf}. The proof for Lemma~\ref{clawek}~is thereby complete. 
\end{proof}

\begin{lem}\label{clauni}
For any $u_0\in \mathcal{H}(g_0,h_0)$ and $T>0$, there exists at most one classical solution to problem~\eqref{eqf}~defined over $G_T$ with initial data $(u_0, g_0, h_0)$.
\end{lem}
\begin{proof}
The proof of this lemma is analogous to that for \cite[Theorem~3.5]{dg2}. For the sake of completeness, we include the details 
here. Assume that problem~\eqref{eqf}~admits two classical solutions $(u_1,g_1,h_1)$ and $(u_2,g_2,h_2)$ defined for $0<t\leq T$ with the same initial data $(u_0, g_0, h_0)$. 
Let $I$ be an open interval such that $I \supset [g_1(T), h_1(T)]\cup [g_2(T), h_2(T)]$ and $\tilde{u}_i$ be defined by \eqref{weakdu} with $(u,g,h)$ replaced by $(u_i, g_i, h_i)$ for $i=1,\,2$. Then $\tilde{u}_1,\,\tilde{u}_2$ are continuous over $\overline{I_T}$, and by  \eqref{dweakf} we obtain
\begin{equation}\label{test}
 \int\limits_{0}^T\!\!\int\limits_{I} \big[ \kappa(\tilde{u}_2)-\kappa(\tilde{u}_1) \big](\partial_t\phi+de\partial_{xx}\phi+el\phi) dxdt=0
 \end{equation}
for every function $\phi \in C^2(\overline{I_T})$ such that $\phi=0$ on $({T}\times I)\cup([0,T]\times \partial I)$, where
\begin{equation*}
l(t,x)=\!\left\{\begin{array}{lll}
\!\! \displaystyle\frac{f(t,x,\tilde{u}_2(t,x))-f(t,x,\tilde{u}_1(t,x))}{\tilde{u}_2(t,x)-\tilde{u}_1(t,x)} & \hbox{if }\tilde{u}_1(t,x)\neq \tilde{u}_2(t,x),\vspace{3pt}\\
\!\!0 & \hbox{if }\tilde{u}_1(t,x)= \tilde{u}_2(t,x),\end{array}\right.
\end{equation*}
and
\begin{equation*}
e(t,x)=\!\left\{\begin{array}{lll}
\!\! \displaystyle\frac{\tilde{u}_2(t,x))-\tilde{u}_1(t,x)}{\kappa(\tilde{u}_2(t,x))-\kappa(\tilde{u}_1(t,x))} & \hbox{if }\tilde{u}_1(t,x)\neq \tilde{u}_2(t,x),\vspace{3pt}\\
\!\!0 & \hbox{if }\tilde{u}_1(t,x)= \tilde{u}_2(t,x).\end{array}\right.
\end{equation*}
By the definition of $\kappa$, one sees that there is some $0<C_1\leq 1$ such that $0\leq e(t,x)\leq C_1$ {\rm a.e.} $(t,x)\in I_T$. 
Since the function $f(t,x,s)$ is of class $C^1$ in $s\geq 0$ uniformly in $(t,x)\in\R\times\R$ and since $\tilde{u}_i(t,x)$ is bounded in $(t,x)\in I_T$ for $i=1,\,2$, the function $l(t,x)$ is bounded in $(t,x)\in I_T$. We then approximate $e$ and $l$ by smooth functions $e_m\in C^{\infty}(\overline{I_T})$ and $l_m\in C^{\infty}(\overline{I_T})$ such that  
\begin{equation}\label{appconverel}
\|e_m-e \|_{L^2(I_T)} \to 0,\,\,\,\, \|l_m-l \|_{L^2(I_T)} \to 0  \,\,\hbox{ as }\,\,m\to\infty,
\end{equation}
and 
\begin{equation}\label{appboundel}
\inf_{I_T} e_m\geq \frac 1m,\; \;\Big\|\frac{e}{e_m}\Big\|_{L^{2}(I_T)}\leq C_2,\,\,\,  \|e_m \|_{L^{\infty}(I_T)}\leq C_2,\,\,\, \|l_m \|_{L^{\infty}(I_T)}\leq C_2 
\end{equation}
for some positive constants $C_2$ independent of $m$ (the existence of such an approximation $e_m$ follows from \cite[Lemma 5]{ch2}). We now fixed a function $q\in C_c^{\infty}(I_T)$. It is well known that the following problem 
\begin{equation*}
\left\{\baa{ll}
\partial_t\phi_m+de_m\partial_{xx}\phi_m+e_ml_m\phi_m=q,& (t,x)\in I_T,\vspace{3pt}\\
\phi_m(T,x)=0,&x\in I,\vspace{3pt}\\
\phi_m(t,x)=0,& 0\leq t\leq T,\, x\in\partial I, \eaa\right.
\end{equation*} 
admits a unique smooth solution $\phi_m$. Moreover, it follows from the proof in \cite[Lemmas~3.6-3.7]{dg2}  that there exists some positive constant $C_3$ independent of $m$ such that 
\begin{equation}\label{unifesti} 
\big\| \phi_m \big\|_{L^{\infty}(I_T)}\leq C_3\quad\hbox{and}\quad  \big\| e_m^{1/2}\partial_{xx}\phi_m \big\|_{L^2(I_T)}\leq C_3.
\end{equation}
Taking each $\phi_m$ as a text function in \eqref{test} gives 
$$\int\limits_{0}^T\!\!\int\limits_{I} \big[ \kappa(\tilde{u}_2)-\kappa(\tilde{u}_1) \big](\partial_t\phi_m+de\partial_{xx}\phi_m+el\phi_m) dxdt=0. $$
This implies that 
\begin{equation*}
\begin{split}
&\int\limits_{0}^T\!\!\int\limits_{I} \big[ \kappa(\tilde{u}_2)-\kappa(\tilde{u}_1) \big]q dxdt\\
 &= \int\limits_{0}^T\!\!\int\limits_{I} \big[ \kappa(\tilde{u}_2)-\kappa(\tilde{u}_1) \big]\big(\partial_t\phi_m+de_m\partial_{xx}\phi_m+e_ml_m\phi_m \big)dxdt\\
&=\int\limits_{0}^T\!\!\int\limits_{I} \big[ \kappa(\tilde{u}_2)-\kappa(\tilde{u}_1) \big]\big\{d(e_m-e)\partial_{xx}\phi_m+(e_ml_m-el)\phi_m  \big\}dxdt.
\end{split}
\end{equation*}
Hence, by the boundedness of $\kappa(\tilde{u}_i)$ for $i=1,\,2$ and the first estimate in \eqref{unifesti}, one has 
\begin{equation*}
\int\limits_{0}^T\!\!\int\limits_{I} \big[ \kappa(\tilde{u}_2)-\kappa(\tilde{u}_1) \big]q dxdt \leq  C_4 \int\limits_{0}^T\!\!\int\limits_{I} 
\big|e_m-e\big|\big|\partial_{xx}\phi_m\big| dxdt + C_5\int\limits_{0}^T\!\!\int\limits_{I} \big|e_ml_m-el\big| dxdt\\
\end{equation*}
for some positive constants $C_4$ and $C_5$ independent of $m$. Then, on the one hand, by the convergences in \eqref{appconverel} and boundedness of $e_m$, $l_m$ in \eqref{appboundel} , one has 
 $$\int\limits_{0}^T\!\!\int\limits_{I} \big|e_ml_m-el\big| dxdt \to 0 \,\,\hbox{ as }\,\,m\to\infty.$$
On the other hand, it follows from the H\"{o}lder inequality that 
\begin{equation*} 
\begin{split}
&\int\limits_{0}^T\!\!\int\limits_{I} \big|e_m-e\big|\big|\partial_{xx}\phi_m\big| dxdt \\
&\leq \Big(\int\limits_{0}^T\!\!\int\limits_{I} \frac{|e_m-e|^2}{|e_m|}dxdt\Big)^{\frac{1}{2}}\Big(\int\limits_{0}^T\!\!\int\limits_{I} |e_m||\partial_{xx}\phi_m|^2 dxdt\Big)^{\frac{1}{2}}\\
&\leq \big\| e_m-e\big\|_{L^2(I_T)}^{\frac{1}{2}} \Big(\int\limits_{0}^T\!\!\int\limits_{I} \frac{|e_m-e|^2}{|e_m|^2}dxdt\Big)^{\frac{1}{4}}\Big(\int\limits_{0}^T\!\!\int\limits_{I} |e_m||\partial_{xx}\phi_m|^2 dxdt\Big)^{\frac{1}{2}}.
\end{split}
\end{equation*}  
Therefore, due to the second inequality in \eqref{appboundel}  and  the second inequality in \eqref{unifesti},  it follows that 
$$\int\limits_{0}^T\!\!\int\limits_{I} \big|e_m-e\big|\big|\partial_{xx}\phi_m\big| dxdt \to 0 \,\,\hbox{ as }\,\,m\to\infty.$$
We thus obtain
$\int_{0}^T\int_{I} \big[ \kappa(\tilde{u}_2)-\kappa(\tilde{u}_1) \big]q dxdt \leq  0$.
Due to the arbitrariness of $q\in C_c^{\infty}(I_T)$, this implies that $\kappa(\tilde{u}_1) = \kappa(\tilde{u}_2)$ {\rm a.e.} in $I_T$. By the definition of $\kappa$, one gets that $\tilde{u}_1 = \tilde{u}_2$ {\rm a.e.} in $I_T$. Since $u_i\in C(\overline{I_T})$ for $i=1,2$, it follows that  $\tilde{u}_1(t,x) = \tilde{u}_2(t,x)$ for all $(t,x)\in I_T$, and hence $g_1(t)=g_2(t)$ and $h_1(t)=h_2(t)$ for every $0<t\leq T$. The proof of Lemma~\ref{clauni} is thereby complete.
\end{proof}
  
Theorem~\ref{existence} clearly follows directly from 
 Lemmas~\ref{lem-existence}~and~\ref{clauni}.

\subsection{Continuous dependence and comparison principle}\label{sec22}
In this section, we first show that the classical solutions obtained in Theorem~\ref{existence} depend continuously on the initial data, and then we prove a comparison principle. 
These results will play important roles in Part 2. 

To prove the continuous dependence, we  introduce a few notations. 
For any $(u_0,g_0,h_0)\in\mathcal{H}(g_0,h_0)\times\R\times\R$, and any sequence $(u_{0n},g_{0n},h_{0n})_{n\in\N}\subset \mathcal{H}(g_{0n},h_{0n})\times\R\times\R$, we say $(u_{0n},g_{0n},h_{0n})$ converges  to $(u_{0},g_0, h_0)$  
 as $n\to\infty$, if 
\[
\mbox{ $g_{0n}\to g_0$, $h_{0n}\to h_0$ and $u_{0n}(x)\to u_{0}(x)$ uniformly in $x\in \R$,}
\]
 where   $u_{0n}$ and $u_0$ are always extended  to $\R$ by taking the value 0 
outside their supporting sets.
 For any fixed $t>0$, the convergence of $\big(u_n(t,x),g_n(t),h_n(t)\big)$ to $\big(u(t,x),g(t),h(t)\big)$  is  defined in a similar way, where $(u_n,g_n,h_n)$ is the solution of \eqref{eqf} with initial data  $(u_{0n}, g_{0n}, h_{0n})$, and $(u,g,h)$ is the solution of \eqref{eqf} with initial data $(u_{0}, g_{0},h_{0})$.

\begin{pro}\label{cdepend}
\begin{itemize}
\item[(i)]
Suppose that $(u_{0n},g_{0n},h_{0n})$ converges to $(u_{0},g_0, h_0)$  as $n\to\infty$. Then for any given $T>0$, $\big(u_n(t,x),g_n(t),h_n(t)\big)$ converges to $\big(u(t,x),g(t),h(t)\big)$  as $n\to\infty$ uniformly in $t\in[0,T]$.

\item[(ii)] Suppose that $ \lim_{n\to\infty}g_{0n}=-\infty$ and $\lim_{n\to\infty}h_{0n}=\infty$ and that $u_{0n}(x)$ converges to $u_0(x)$ locally uniformly in $x\in\R$. Then for any given $T>0$,  $u_n(t,x)$ converges to $v(t,x;u_0)$ locally uniformly in $x\in\R$ and uniformly in $t\in [0,T]$, where $v(t,x;u_0)$ is the unique solution of the Cauchy problem \eqref{cauchy} with initial datum $v(0,\cdot)=u_0(\cdot)$ in $\R$. 
\end{itemize}
\end{pro}

\begin{proof}
We only present the proof for the first statement, since the proof for the second one is similar and even simpler. 

Since $(u_{0n}, g_{0n}, h_{0n})\to (u_0, g_0, h_0)$ as $n\to\infty$, we can find $(\underline u_{0n}, \underline g_{0n}, \underline h_{0n})$ and $(\overline u_{0n}, \overline g_{0n}, \overline h_{0n})$ such that, for every $n\in\N$,
\[
\underline u_{0n}\in C^2([\underline g_{0n}, \underline h_{0n}]),\; \underline u_{0n}(\underline g_{0n})=\underline u_{0n}(\underline h_{0n})
=0,\; \underline u_{0n}(x)>0 \mbox{ for } x\in (\underline g_{0n}, \underline h_{0n}),
\]
\[
\overline u_{0n}\in C^2([\overline g_{0n}, \overline h_{0n}]),\; \overline u_{0n}(\overline g_{0n})=\overline u_{0n}(\overline h_{0n})
=0,\; \overline u_{0n}(x)>0 \mbox{ for } x\in (\overline g_{0n}, \overline h_{0n}),
\]
\[
\underline u_{0n}\leq u_{0n}\leq \overline u_{0n} \mbox{ in } \R,
\;\;
\underline g_{0n}\geq g_{0n}\geq \overline g_{0n},\; \underline h_{0n}\leq h_{0n}\leq \overline h_{0n},
\]
\[
\underline g_{0n}\searrow g_0,\;  \overline g_{0n}\nearrow g_0,
\;\;
 \underline h_{0n}\nearrow h_0,\;  \overline h_{0n}\searrow h_0 \mbox{ as } n\to\infty,
\]
and
\[
 \underline u_{0n}\nearrow u_0,\;  \overline u_{0n}\searrow u_0 \mbox{ uniformly in } \R \mbox{ as } n\to\infty.
\]
Here, as before, the initial functions are extended  to $\R$ by the value 0 outside their supporting sets.

Let $(\underline u_n, \underline g_n,\underline h_n)$ be the unique classical solution of \eqref{eqf} with initial data $(\underline u_{0n}, \underline g_{0n},  \underline h_{0n})$,  and $(\overline u_{n}, \overline g_{n}, \overline h_{n})$ be the unique solution of \eqref{eqf} with initial data 
$(\overline u_{0n}, \overline g_{0n}, \overline h_{0n})$.  It follows from the proof of \cite[Lemma 3.5]{dlin} that 
 $$\overline g_n(t)\leq g_n(t)\leq \underline g_n(t),\,\,\, \overline h_n(t)\geq h_n(t)\geq \underline h_n(t)  \,\hbox{ in }\, (0,T], $$
 and
$$ \overline u_n(t,x) \geq u_n(t,x)\geq \underline u_n(t,x)\,\hbox{ for } \, 0<t\leq T,\,\,\, x\in\R,$$  
where $\overline u_n(t,\cdot)$, $u_n(t,\cdot)$ and $\underline u_n(t,\cdot)$ are extended to all of $\R$ by taking the value 0 outside their supporting sets.

By Theorem \ref{existence} and Remark \ref{approx-above}, we know that
\[
\lim_{n\to\infty}\underline g_n(t)= \lim_{n\to\infty}\overline g_n(t)=g(t),\; \lim_{n\to\infty}\underline h_n(t)= \lim_{n\to\infty}\overline h_n(t)=h(t)
\]
uniformly in $t\in [0, T]$. It follows that
\[
\lim_{n\to\infty} g_n(t)=g(t),\;  \lim_{n\to\infty} h_n(t)=h(t) \mbox{ uniformly in } t\in [0, T].
\]
Moverover, from
\[
\lim_{n\to\infty} \underline u_n(t,x)=\lim_{n\to\infty} \overline u_n(t,x)=u(t,x) \mbox{ in } C_{loc}(G_T),
\]
and \eqref{loces} and \eqref{g-h-es}, we see that
\[
\lim_{n\to\infty} \underline u_n(t,x)=\lim_{n\to\infty} \overline u_n(t,x)=u(t,x)
\]
uniformly in $[\tau, T]\times \R$ for any $\tau\in (0, T)$.

Furthermore, by the proof of Lemma \ref{ini-cond}, we easily see that
\[
\lim_{n\to\infty, t\to 0}\overline u_n(t,x)=\lim_{n\to\infty, t\to 0}\underline u_n(t,x)=u_0(x) \mbox{ in } L^\infty(\R).
\]

Combining the above conclusions, we see that
\[
\lim_{n\to\infty}  u_n(t,x)=u(t,x)
\]
uniformly in $x\in\R,\; t\in [0, T]$. 
\end{proof}

Having in hand the above continuous dependence, we now establish the following comparison principle for problem~\eqref{eqf}~with initial function belonging to $\mathcal{H}(g_0,h_0)$, which is an easy extension of that for~\eqref{eqf}~with $C^2$ initial functions.  

\begin{pro}\label{comparison}
Suppose that $T\in (0,\infty)$, that $\tilde{g},\,\tilde{h}\in C\big([0,T]\big)\cap C^1\big((0,T]\big)$ and that $\tilde{u}\in C\big(\overline{\tilde{D}_T}\big)\cap C^{1,2}\big(\tilde{D}_T\big)$ with $\tilde{D}_T=\big\{(t,x)\in\R^2: \, 0<t\leq T,\, \tilde{g}(t)\leq x\leq \tilde{h}(t)\big\}$. 
\begin{itemize}
\item[(i)] If
\begin{equation}\label{ll}
\left\{\baa{ll}
\tilde{u}_t \geq d\tilde{u}_{xx}+f(t,x,\tilde{u}),& 0<t\leq T,\,\,\,\tilde{g}(t)<x<\tilde{h}(t),\vspace{3pt}\\
\tilde{u}(t,\tilde{g}(t))=0,\,\,\,\, \tilde{g}'(t)\leq -\mu \tilde{u}_x(t,\tilde{g}(t)),&0<t\leq T, \,\,\, x=\tilde{g}(t),\vspace{3pt}\\
\tilde{u}(t,\tilde{h}(t))=0,\,\,\,\tilde{h}'(t)\geq -\mu \tilde{u}_x(t,\tilde{h}(t)),&0<t\leq T,\,\,\, x=\tilde{h}(t),\eaa\right.
\end{equation} 
and
$$[g_0,h_0]\subset [\tilde{g}(0), \tilde{h}(0)],\quad\quad u_0(x)\leq \tilde{u}(0,x)\,\hbox{ in } \, [g_0,h_0],$$
 then the solution $(u,g,h)$ of problem~\eqref{eqf}~ with initial data $(u_0, g_0, h_0)$ satisfies
 $$g(t)\geq \tilde{g}(t),\,\,\,h(t)\leq \tilde{h}(t) \,\hbox{ in }\, (0,T], $$ 
 and
 $$ u(t,x)\leq \tilde{u}(t,x)\,\hbox{ for } \, 0<t\leq T,\,\,\,g(t)\leq x\leq h(t).$$  
\item[(ii)] If the inequalities in \eqref{ll} are reversed, and 
$$
[g_0,h_0]\supset [\tilde{g}(0), \tilde{h}(0)]\quad\hbox{and}\quad u_0(x)\geq \tilde{u}(0,x)\,\hbox{ in } \, [\tilde g(0),\tilde h(0)],$$
 then the solution $(u,g,h)$ of problem~\eqref{eqf}~with initial data $(u_0, g_0, h_0)$ satisfies
 $$g(t)\leq \tilde{g}(t),\,\,\,h(t)\geq \tilde{h}(t) \,\hbox{ in }\, (0,T], $$ 
 and
 $$ u(t,x)\geq \tilde{u}(t,x)\,\hbox{ for } \, 0<t\leq T,\,\,\, \tilde g(t)\leq x\leq \tilde h(t).$$ 
\end{itemize}
\end{pro}
\begin{proof} We only give the proof for part (i), as part (ii) can be proved analogously.
Choose sequences $(g_{0n})_{n\in\N}\subset \R$, $(h_{0n})_{n\in\N}\subset \R$ such that $g_{0n}$ decreases to $g_0$, $h_{0n}$ increases to $h_0$ as $n\to\infty$ and $(u_{0n})_{n\in\N}\subset C^2\big([g_{0n},h_{0n}]\big)$ such that
$$0<u_{0n}(x)\leq u_0(x) \,\hbox{ in }\, [g_{0n},h_{0n}]\,\,\, \hbox{and}\,\,\,  u_{0n}(g_{0n})= u_{0n}(h_{0n})=0\,\hbox{ for each } \, n\in\N,$$
 and that $u_{0n}$ converges to $u_{0}$ as $n\to\infty$ uniformly in $[g_0,h_0]$. For each $n\in\N$, let $(u_n,g_n,h_n)$ be the calssical solution of problem~\eqref{eqf}~with initial data $(u_{0n}, g_{0n},h_{0n})$. It then follows from the proof of \cite[Lemma 3.5]{dlin} that 
 $$g_n(t)\geq \tilde{g}(t),\,\,\,h_n(t)\leq \tilde{h}(t) \,\hbox{ in }\, (0,T], $$
 and
$$ u_n(t,x)\leq \tilde{u}(t,x)\,\hbox{ for } \, 0<t\leq T,\,\,\,g_n(t)\leq x\leq h_n(t).$$  
Due to Proposition~\ref{cdepend}, one can pass to the limit $n\to\infty$ in the above inequalities, and obtain all the required conclusions. 
\end{proof}

\subsection{Parallel results for an auxiliary problem}\label{sec23}
In order to prove the existence of spreading speeds for problem~\eqref{eqf} in Part 2, we need to study the following auxiliary problem  
\begin{equation}\label{eqfunbd}\left\{\baa{ll}
u_t=du_{xx}+f(t,x,u),& -\infty<x<h(t),\quad t>0,\vspace{3pt}\\
u(t,h(t))=0,\,\,\, h'(t)=-\mu u_x(t,h(t)),&t>0, \vspace{3pt}\\
h(0)=h_0,\quad u(0,x)=u_0(x),& -\infty< x\leq h_0,\eaa\right.
\end{equation} 
with initial data $u_0\in {\mathcal{H}_+}(h_0)$, where
\begin{equation*}
{\mathcal{H}_+}(h_0):=\Big\{\phi\in C\big((-\infty,h_0]\big)\cap L^{\infty}\big((-\infty,h_0]\big):\,\phi(h_0)=0, \,\phi(x)>0 \hbox{ in }(-\infty,h_0) \Big\}.
\end{equation*}

All the results in the previous two subsections carry over to this problem without difficulties. Here we list these corresponding results while leaving their proofs to the interested reader.

\begin{theo}\label{existenceunbd}
Suppose that \eqref{zero} and \eqref{globalb} are satisfied. For any $u_0\in\mathcal{H}_+(h_0)$, problem \eqref{eqfunbd} admits a unique classical solution $\big(u(t,x),h(t)\big)$ defined for all $t>0$, and $h\in C^{1}\big((0,+\infty)\big)\cap C\big([0,\infty)\big)$, $u\in C^{1,2}(G_+)\cap C\big(\overline{G_+}\big)$ with $G_+=\big\{(t,x)\in\R^2:\,t\in (0,\infty),\,x\in (-\infty,h(t)] \big\}$.
Furthermore, for any $T>\tau>0$ and any $A\leq h_0$, there holds
\begin{equation*}
\big\|u\big\|_{C^{(1+\alpha)/2,1+\alpha}(G_{A,T}^{\tau})}+\big\|h\big\|_{C^{1+\alpha/2}([\tau,T])} \leq C,
\end{equation*}
where $G_{A,T}^\tau=\big\{(t,x)\in\R^2:\,t\in[\tau,T],\,x\in[A,h(t)] \big\}$, and $C$ is a positive constant depending on $\tau$, $T$, $f$ and $\|u_0\|_{L^{\infty}((-\infty,h_0])}$. 
\end{theo}

\begin{rem}\label{runilip}
{\rm  We should remark that, for any given $u_0\in\mathcal{H}_+(h_0)$, let $\big(u(t,x),h(t)\big)$ be the unique solution of \eqref{eqfunbd} with initial datum $u(0,x)=u_0(x)$ in $(-\infty,h_0)$, then for $T>0$, $u(T,x)$ is Lipschitz continuous in $(-\infty,h(T)]$. It follows from the estimate in Theorem~\ref{existenceunbd} that the Lipschitz constant only depends on  $T$, $f$ and $\|u_0\|_{L^{\infty}((-\infty,h_0)]}$.}
\end{rem}

\begin{pro}\label{cdependunbd} Under the assumptions \eqref{zero} and \eqref{globalb}, the following conclusions hold.
\begin{itemize}
\item[(i)] For any given $h_0>0$ and any given sequence $(h_{0n})_{n\in\N}\subset \R^+$, let $u_0\in\mathcal{H}_+(h_0)$ and $u_{0n}\in\mathcal{H}_+(h_{0n})$. Suppose that $(u_{0n},h_{0n})$ converges to $(u_{0}, h_0)$ in $C_{loc}\big((-\infty,h_0]\big)\times\R$ as $n\to\infty$. Then for any given $T>0$, $\big(u_n(t,x), h_n(t)\big)$ converges to $\big(u(t,x),h(t)\big)$ in $C_{loc}\big((-\infty,h(t)]\big)\times\R$ as $n\to\infty$ unfiormly in $t\in[0,T]$, where $(u_n,h_n)$ is the solution for \eqref{eqfunbd} with $u_n(0,\cdot)=u_{0n}(\cdot)$ in $(-\infty,h_{0n}]$, and $(u,h)$ is the solution for  \eqref{eqfunbd} with $u(0,\cdot)=u_{0}(\cdot)$ in $(-\infty,h_{0}]$.

\item[(ii)]
In addition to the assumptions in {\rm (i)}, if $h_0=+\infty$, then for any given $T>0$,  $u_n(t,x)$ converges to $v(t,x;u_0)$ locally uniformly in $x\in\R$ and uniformly in $t\in[0,T]$, where $v(t,x;u_0)$ is the solution of the Cauchy problem \eqref{cauchy} with initial datum $v(0,\cdot)=u_0(\cdot)$ in $\R$. 
\end{itemize}
\end{pro}

\begin{pro}\label{comparisonunbd}
Suppose that $T\in (0,\infty)$, that $\tilde{h}\in C\big([0,T]\big)\cap C^1\big((0,T]\big)$ and that $\tilde{u}\in C\big(\overline{\tilde{D}_{+,T}}\big)\cap C^{1,2}\big(\tilde{D}_{+,T}\big)$ with $\tilde{D}_{+,T}=\big\{(t,x)\in\R^2: \, 0<t\leq T,\, -\infty< x\leq \tilde{h}(t)\big\}$. 
\begin{itemize}
\item[(i)] If 
\begin{equation*}\left\{\baa{ll}
\tilde{u}_t \geq d\tilde{u}_{xx}+f(t,x,\tilde{u}),& 0<t\leq T,\,\,\,-\infty<x<\tilde{h}(t),\vspace{3pt}\\
\tilde{u}(t,\tilde{h}(t))= 0,\,\,\,\, \tilde{h}'(t)\geq -\mu \tilde{u}_x(t,\tilde{h}(t)),&0<t\leq T,\eaa\right.
\end{equation*} 
and
$$h_0\leq \tilde{h}(0)\quad\hbox{and}\quad u_0(x)\leq \tilde{u}(0,x)\,\hbox{ in } \, (-\infty,h_0],$$
 then the solution $(u,h)$ of problem \eqref{eqfunbd} satisfies
$$h(t)\leq \tilde{h}(t) \,\hbox{ in }\, (0,T]\,\hbox{ and }\,  u(t,x)\leq \tilde{u}(t,x)\,\hbox{ for } \, 0<t\leq T,\,\,\,-\infty<x \leq h(t).$$ 
\item[(ii)]
If in the assumptions of part {\rm (i)} all the inequalities are reversed, then the solution $(u,h)$ of problem \eqref{eqfunbd} satisfies
$$h(t)\geq \tilde{h}(t) \,\hbox{ in }\, (0,T]\,\hbox{ and }\,  u(t,x)\geq \tilde{u}(t,x)\,\hbox{ for } \, 0<t\leq T,\,\,\,-\infty<x \leq \tilde h(t).$$ 
\end{itemize}
\end{pro}   

The pair of functions $(\tilde{u},\tilde{h})$ in part (i) of Proposition~\ref{comparisonunbd} is often called an upper 
solution for problem \eqref{eqfunbd}, and in part (ii) it is called a lower solution.

Lastly we note that each of the above listed results for problem~\eqref{eqfunbd} has a paralelle version for the following problem 
\begin{equation}\label{eqfunbdn}\left\{\baa{ll}
u_t=du_{xx}+f(t,x,u),& g(t)<x<\infty,\quad t>0,\vspace{3pt}\\
u(t,g(t))=0,\,\,\, g'(t)=-\mu u_x(t,g(t)),&t>0, \vspace{3pt}\\
g(0)=g_0,\quad u(0,x)=u_0(x),& g_0 \leq x < \infty,\eaa\right.
\end{equation} 
with initial data $u_0\in {\mathcal{H}_-}(g_0)$, where
\begin{equation*}
{\mathcal{H}_-}(g_0):=\Big\{\phi\in C\big([g_0,\infty)\big)\cap L^{\infty}\big([g_0,\infty)\big):\,\phi(g_0)=0, \,\phi(x)>0 \hbox{ in }(g_0,\infty) \Big\}.
\end{equation*}


\section{Spreading-vanishing dichotomy}\label{sec3}  
This section is devoted to the proof of Theorems~\ref{spva1}~and~\ref{spva2}, on the spreading-vanishing dichotomy and sharp criteria for spreading or vanishing. 

Throughout this section, the function $f(t,x,u)$ is supposed to satisfy the assumptions \eqref{zero}, \eqref{hyp2}, \eqref{period},
 \eqref{hyp1} and \eqref{monostable}. The arguments in this section mainly follow those used in \cite{dg1, dgp}, where 
similar free boundary problems in homogeneous, or  time-periodic media were considered in a radially smmetric setting. 
In order not to repeat the arguments in \cite{dg1, dgp}, in what follows, we only  provide the details when considerable 
changes are needed. 

For any fixed $y\in\R$ and any fixed $R>0$, let $\lambda_{1,R}^y$ be the real number $\lambda$ such that  there exists a $C^{1,2}(\R\times[-R,R])$ function $\psi$ satisfying 
\begin{equation}\label{localeig}
\left\{\baa{l}
\partial_t\psi-d\partial_{xx}\psi-\partial_{u}f(t,x+y,0)\psi=\lambda\psi\,\hbox{ in } \R\times(-R,R),\vspace{3pt}\\
\psi>0 \,\,\hbox{ in } \,\, \R\times (-R,R),\vspace{3pt}\\
\psi(t,-R)=\psi(t,R)=0 \,\,\hbox{ for  all } \,\,t\in\R,\vspace{3pt}\\
\psi(t,x)\hbox{ is $\omega$-periodic in t}.\eaa\right.
\end{equation}  
It is well known (see \cite{he}) that this real number $\lambda_{1,R}^y$ is  the principal eigenvalue of \eqref{localeig}, which exists uniquely, and $\psi$ is the 
(unique up to scalar multiplication) corresponding eigenfunction.  Furthermore, we have the following properties of $\lambda_{1,R}^y$.

\begin{lem}\label{proprin}
Let $\lambda_{1,R}^y$ be the principal eigenvalue of  \eqref{localeig}. Then $\lambda_{1,R}^y$ is continuous in $(y,R)\in\R\times(0,\infty)$, and $\lambda_{1,R}^y$ is $L$-periodic in $y$ and  strictly decreasing in $R>0$. Moreover, for any fixed $y\in\R$, there holds 
$$\lim_{R\to\infty} \lambda_{1,R}^y = \lambda_1(\mathcal{L})\quad \hbox{and}\quad \lim_{R\to0} \lambda_{1,R}^y=\infty,$$
where $\lambda_1(\mathcal{L})$ is the generalized principal eigenvalue given in \eqref{princi}.
\end{lem} 

\begin{proof}
 We first prove that $\lambda_{1,R}^y$ is continuous in $(y,R)\in\R\times(0,\infty)$. For any given 
$(y,R)\in\R\times(0,\infty)$, let $\psi^y_R(t,x)>0$ be the principal eigenfunction corresponding to 
$\lambda_{1,R}^y$, normalized by   $\|\psi^y_R\|_{L^{\infty}(\R\times[-R,R])}=1$. 
Set $\varphi^y_R(t,x)=\psi^y_R(R^2t,Rx)$. Then $(\lambda, \psi)=(R^2\lambda_{1,R}^y, \varphi^y_R(t,x))$ 
is an eigen pair to the following eigenvalue problem
\begin{equation}
\left\{\baa{l}\label{changef}
\partial_t\psi-d\partial_{xx}\psi-\mu(t,x,y,R)\psi=\lambda \psi\,\,\,\hbox{ in } \,\,\R\times(-1,1),\vspace{3pt}\\
 \psi>0 \,\,\hbox{ in } \,\, \R\times (-1,1),\vspace{3pt}\\
\psi(t,-1)=\psi(t,1)=0 \,\,\hbox{ for  all } \,\,t\in\R,\vspace{3pt}\\
 \psi(t,x)\hbox{ is $\omega/R^2$-periodic in t},\eaa\right.
\end{equation} 
with
\[
\mu(t,x,y,R):=R^2\partial_u f(R^2t, Rx+y,0).
\]

Let us observe that, if we denote by 
 $\tilde{\lambda}_{1}(\mu)$ the principal eigenvalue of \eqref{changef}, then 
\[
\tilde\lambda_1(\mu)=\tilde \lambda_1(R^2\partial_u f(R^2t, Rx+y,0))=R^2\lambda_{1,R}^y,
\]
and if $\mu$ is replaced by a constant $\mu_0$, then
\[
\tilde\lambda_1(\mu_0)=\lambda^*_1-\mu_0,
\]
where $\lambda^*_1>0$ is the principal eigenvalue of the problem
$$
-d\varphi''=\lambda\varphi\,\hbox{ in } \,(-1,1);\,\,
\varphi(-1)=\varphi(1)=0.
$$

By the monotonicity of $\tilde\lambda_1(\mu)$ on $\mu$, we obtain
\begin{equation}\label{lamcom}
\lambda_1^*-R^2m^*=\tilde{\lambda}_1(R^2m^*)\leq\tilde{\lambda}_1(\mu)=
R^2\lambda_{1,R}^y\leq\tilde{\lambda}_1(R^2m_*)=\lambda^*_1-R^2m_*
\end{equation}
 where
\[
m_*:= \min_{ (t,x)\in\R^2}\partial_{u}f(t,x,0),\; m^*:=\max_{ (t,x)\in\R^2}\partial_{u}f(t,x,0).
\]
Therefore, for any finite closed interval $I\subset (0,\infty)$, $\tilde\lambda_1(\mu)=R^2  \lambda_{1,R}^y $ is bounded in 
$(y,R)\in\R\times I$. Furthermore, since the principal eigenvalue $\tilde{\lambda}_{1}(\mu)$ is unique, it then follows from standard parabolic estimates and a compactness argument that $\tilde{\lambda}_{1}(\mu)$ is uniformly continuous in $(y,R)\in\R\times I$. Thus, $\lambda_{1,R}^y$ is continuous in $(y,R)\in\R\times(0,\infty)$.

Since the function $\partial_{u}f(t,x+y,0)$ is $L$-periodic in $y$, by the uniqueness of the principal eigenvalue $\lambda_{1,R}^y$, it is obvious  that $\lambda_{1,R}^y$ is $L$-periodic in $y$.

Next, it follows from \cite[proposition 3.2]{na1} and \cite[Theorem 2.6]{na3} that $\lambda_{1,R}^y$ is nonincreasing in $R>0$ and 
converges to $\lambda_1(\mathcal{L})$ uniformly in $y\in\R$ as $R\to\infty$. Moreover, by similar arguments to those used in \cite[Lemma 3.5]{bhr1}, one concludes that $\lambda_{1,R}^y$ is strictly decreasing in $R>0$. 

Finally, we consider the convergence of  $\lambda_{1,R}^y$ as $R\to 0$. By \eqref{lamcom} we obtain
\[
\lim_{R\to 0} R^2\lambda_{1,R}^y=\lambda_1^*>0,
\]
which clearly implies 
 $\lim_{R\to0} \lambda_{1,R}^y=\infty$. The proof of Lemma~\ref{proprin} is thereby complete.
\end{proof}
 
In view of Lemma~\ref{proprin} and the assumption that $\lambda_1(\mathcal{L})<0$ in \eqref{monostable}, it follows that for any $y\in\R$, there exists a unique $R^*=R^*(y)$ such that 
\begin{equation}\label{lambda*}
\lambda_{1,R^*}^y=0\,\,\,\hbox{ and } \,\,\, \lambda_{1,R}^y<0 \, \hbox{ for } \,R>R^*,\quad\lambda_{1,R}^y>0\, \hbox{ for }R<R^*.
\end{equation}
Furthermore, the function $y\mapsto R^*(y)$ is continuous and $L$-periodic in $y\in\R$. 
Moreover, one has the following property.

\begin{lem}\label{localsta}
For any given $y\in\R$ and any given $R>R^*(y)$, the following problem 
\begin{equation}\label{localstaeq}
\left\{\baa{l}
\partial_tp-d\partial_{xx}p=f(t,x+y,p)\,\hbox{ in } \R\times(-R,R),\vspace{3pt}\\
p(t,-R)=p(t,R)=0 \,\,\hbox{ for  all } \,\,t\in\R,\eaa\right.
\end{equation}  
admits a unique positive time $\omega$-periodic solution $p_{R,y}\in C^{1,2}(\R\times[-R,R])$. Moreover, $p_{R,y}$ is globally asymptotically stable in the sense that for any nonnegative non-null function $\tilde{u}_0\in C([-R,R])$ with $\tilde{u}_0(-R)=\tilde{u}_0(R)=0$, there holds 
$$u_{R,y}(t+s,x;\tilde{u}_0) - p_{R,y}(t+s,x)\to 0 \,\,\hbox{ as }\,\, s\to\infty\,\,\hbox{ in }\,\, C^{1,2}_{loc}\big(\R\times[-R,R]\big),$$
where $u_{R,y}(t,x;\tilde{u}_0)$ is the unique solution of the following problem  
\begin{equation*}
\left\{\baa{l}
\partial_tu-d\partial_{xx}u=f(t,x+y,u)\,\hbox{ for }\, t>0,\,\,-R<x<R,\vspace{3pt}\\
u(t,-R)=u(t,R)=0 \,\,\hbox{ for  all } \,\,t>0,\vspace{3pt}\\
u(0,x)=\tilde{u}_0(x)\,\,\hbox{ for } \,\,-R\leq x\leq R.\eaa\right.
\end{equation*}
\end{lem}
\begin{proof}
Let $\lambda_{1,R}^y$ be the principal eigenvalue of problem \eqref{localeig}. Since $R>R^*(y)$, it 
is easy to see that $\lambda_{1,R}^y<0$. This together with the assumptions \eqref{hyp2} and \eqref{hyp1} 
imply all the conclusions of this lemma. The proof is almost identical to that of \cite[Theorem 28.1]{he}, so we omit the details. 
\end{proof} 

The above existence, uniqueness and stability results for problem \eqref{localstaeq} in a bounded domain can be extended to the following problem with an unbounded domain.

\begin{lem}\label{halfsta}
The  problem 
\begin{equation}\label{halfstaeq}
\left\{\baa{l}
\partial_t p_+ -d\partial_{xx} p_+=f(t,x,p_+)\,\hbox{ in } \R\times(-\infty,0),\vspace{3pt}\\
p_+(t,0)=0 \,\,\hbox{ for  all } \,\,t\in\R,\eaa\right.
\end{equation}  
admits a unique positive time $\omega$-periodic solution $p_+ \in C^{1,2}\big(\R\times(-\infty,0]\big)$. 
Moreover, $p_+$ is globally asymptotic stable in the sense that for any nonnegative non-null function $\bar{u}_0\in C\big((-\infty,0]\big)$ with $\bar{u}_0(0)=0$, there holds 
$$u_+(t+s,x;\bar{u}_0)\to p_+(t+s,x)\,\,\hbox{ as }\,\, s\to\infty\,\,\hbox{ in }\,\, C^{1,2}_{loc}\big(\R\times(-\infty,0]\big),$$
where $u_+(t,x;\bar{u}_0)$ is the unique solution of the following problem 
\begin{equation*}
\left\{\baa{l}
\partial_tu-d\partial_{xx}u=f(t,x,u)\,\hbox{ for }\, t>0,\,\,-\infty<x<0,\vspace{3pt}\\
u(t,0)=0 \,\,\hbox{ for  all } \,\,t>0,\vspace{3pt}\\
u(0,x)=\bar{u}_0(x)\,\,\hbox{ for } \,\,-\infty< x\leq 0.\eaa\right.
\end{equation*}
\end{lem}

\begin{proof}
We only prove the existence of positive time periodic solution $p_+$ for problem \eqref{halfstaeq}, since the uniqueness and global asymptotic stability for $p_+$ follows from similar lines to those used in \cite{na3} for problem \eqref{psteady11}, due to the assumptions \eqref{hyp2}, \eqref{period}, \eqref{hyp1} and \eqref{monostable}.

For any $R>2\overline{R}$ where 
\begin{equation}\label{deoverr}
\overline{R}=\max_{y\in\R} R^*(y),
\end{equation}
it follows from Lemma~\ref{localsta} that the following problem
\begin{equation}\label{localeqex}
\left\{\baa{l}
\partial_tp_R-d\partial_{xx}p_R=f(t,x,p_R)\,\hbox{ in } \R\times(-R,0),\vspace{3pt}\\
p_R(t,x)\, \hbox{ is } \omega\hbox{-peirodic in } \,t\in\R,\vspace{3pt}\\ 
p_R(t,-R)=p_R(t,0)=0 \,\,\hbox{ for  all } \,\,t\in\R,\eaa\right.
\end{equation}  
has a unique positive solution $p_R(t,x)\in C^{1,2}\big(\R\times[-R,0]\big)$. Now we show that, for any $R_2>R_1>2\overline{R}$, there holds
\begin{equation}\label{comlocal}
p_{R_2}(t,x)\geq p_{R_1}(t,x)\,\hbox{ for all }\, (t,x)\in \R\times[-R_1,0].
\end{equation}
To do so, we choose a non-null nonnegative function $u_0\in C\big([-R_1,0]\big)$ with $u_0(0)=u_0(-R_1)=0$ such that $u_0(x)\leq p_{R_2}(0,x)$ for all $x\in[-R_1,0]$. Then one sees that $p_{R_2}(t,x)$ is a supersolution to the problem 
\begin{equation}\label{R1}
\left\{\baa{l}
\partial_tu_{R_1}-d\partial_{xx}u_{R_1}=f(t,x,u_{R_1})\,\hbox{ in }\, t>0,\,\,-R_1<x<0,\vspace{3pt}\\
u_{R_1}(t,-R_1)=u_{R_1}(t,0)=0 \,\,\hbox{ for  all } \,\,t>0,\vspace{3pt}\\
u_{R_1}(0,x)=u_0(x)\,\,\hbox{ in } \,\,-R_1\leq x\leq 0.\eaa\right.
\end{equation}
It follows from the parabolic maximum principle that 
$$u_{R_1}(t+n\omega,x)\leq p_{R_2}(t,x) \, \hbox{ for all } t>0,\, -R_1\leq x\leq 0,\,n\in\N.$$
Since $p_{R_1}$ is a globally asymptotically stable solution of problem \eqref{localeqex} with $R=R_1$, we have
$$\lim_{n\to\infty} u_{R_1}(t+n\omega,x)=p_{R_1}(t,x)\,\hbox{ for } t>0,\; -R_1\leq x\leq 0.$$
Hence \eqref{comlocal} holds.

Choose a sequence $\{R_i\}_{i\in\N} \subset [R_0,\infty)$ with $R_i\nearrow \infty$ as $i\to\infty$, and let $p_{R_i}(t,x)\in C^{1,2}\big(\R\times[-R_i,0]\big)$ be the positive solution to \eqref{localeqex} with $R=R_i$.
Since $f(t,x, M)\leq 0$, the positive constant $M$ is a super solution 
 to the equation satisfied by $p_{R_i}(t,x)$. Therefore by the above arguments we obtain
\[
p_{R_i}(t,x)\leq p_{R_{i+1}}(t,x)\leq M \mbox{ for } t\in\R, \; x\in [-R_i, 0],\; i\in\N.
\]
Hence we can define
\[
p_+(t,x):=\lim_{i\to\infty} p_{R_i}(t,x) \mbox{ for } t\in\R, x\in (-\infty, 0].
\]

For any $R_0>2\overline{R}$, by  parabolic estimates to problem \eqref{localeqex} with $R> R_0$ over the domain $[0,\omega]\times[-R_0,0]$, and a standard diagonal process, we see that
$$\lim_{i\to\infty}p_{R_i}(t,x)= p_{+}(t,x)\, \hbox{ in } \, C_{loc}^{1,2}\big(\R\times(-\infty,0]\big),$$
and hence  $p_{+}(t,x)$ is a positive solution to problem \eqref{halfstaeq}. The proof of Lemma~\ref{halfsta} is thereby complete. 
\end{proof}

Now we give the proof for Theorem~\ref{spva1}. In the sequel, $\big(u(t,x),g(t),h(t)\big)$ always denotes
 the unique solution of problem \eqref{eqf} with given initial datum $u_0\in\mathcal{H}(g_0,h_0)$, and $h_{\infty}$, $g_{\infty}$ are the limits of the functions $h(t)$ and $g(t)$ as $t\to\infty$, respectively.

\begin{lem}\label{vanishing}
If $h_{\infty}<\infty$ or $g_{\infty}>-\infty$, then both $h_{\infty}$ and $g_{\infty}$ are finite, and $h_{\infty}-g_{\infty}\leq 2R^*(y_\infty)$ where $y_\infty=(h_\infty+g_\infty)/2$. Moreover,
$$\lim_{t\to\infty} \max_{g(t)\leq x\leq h(t)} u(t,x)=0.$$
\end{lem}
\begin{proof}
Without loss of generality, we assume that $h_{\infty}<\infty$, and we proceed to show $h_{\infty}-g_{\infty}\leq 2R^*(y_{\infty})$. The 
proof for the case $g_{\infty}>-\infty$ is parallel. 

We first show that $g_{\infty}>-\infty$. Assume by contraction that $g_{\infty}=-\infty$. Let $T_0$ be the real number such that $h(t)-g(t)>2\overline{R}$ for all $t\geq T_0$, where $\overline{R}$ is given in \eqref{deoverr}.  It follows from Lemma~\ref{localsta} that, for any fixed $T>T_0$, the following problem 
\begin{equation*}
\left\{\baa{l}
\partial_t(w_T)-d\partial_{xx}w_T=f(t,x,w_T)\,\hbox{ in }\, t>0,\,\,g(T)<x<h(T),\vspace{3pt}\\
w_T(t,g(T))=w_T(t,h(T))=0 \,\,\hbox{ for  all } \,\,t>0,\vspace{3pt}\\
w_T(0,x)=u(T,x)\,\,\hbox{ in } \,\,g(T)\leq x\leq h(T),\eaa\right.
\end{equation*}
has a unique solution $w_T(t,x)\in C^{1,2}\big(\R\times[g(T),h(T)]\big)$, and
 $$w_T(t+s,x)- \bar{w}_T(t+s,x)\to 0 \,\hbox{ as } s\to\infty  \,\hbox{ in } C^{1,2}_{loc}\big(\R\times[g(T),h(T)]\big),$$
where $\bar{w}_T$ is the unique positive time $\omega$-periodic solution to the problem
\begin{equation*}
\left\{\baa{l}
\partial_t\bar{w}_T-d\partial_{xx}\bar{w}_T=f(t,x,\bar{w}_T)\,\hbox{ in } t>0,\, g(T)<x<h(T),\vspace{3pt}\\
\bar{w}_T(t,g(T))=\bar{w}_T(t,h(T))=0 \,\,\hbox{ for  all } \,\,t\in\R.\eaa\right.
\end{equation*}  
By the parabolic maximum principle, one has $u(t+T,x)\geq w_T(t,x)$ for all $t>0$, $g(T)\leq x\leq h(T)$, whence
$$\liminf_{n\to\infty}u(t+n\omega,x) \geq \bar{w}_T(t,x)  \,\hbox{ for all } t>0, \, g(T)\leq x\leq h(T).$$
One the other hand, let $\tilde{u}_0$ be the function in $C((-\infty, h_{\infty}])$ given by
 $\tilde{u}_0(x)=u(T,x)$ for $x\in[g(T),h(T)]$ and $\tilde{u}_0(x)=0$ for $x\in (-\infty,h_{\infty}])\setminus [g(T),h(T)]$. It follows from Lemma~\ref{halfsta} that the following problem 
\begin{equation*}
\left\{\baa{l}
\partial_tw-d\partial_{xx}w=f(t,x,w)\,\hbox{ in }\, t>0,\,\,-\infty<x<h_{\infty},\vspace{3pt}\\
w(t,h_{\infty})=0 \,\,\hbox{ for  all } \,\,t>0,\vspace{3pt}\\
w(0,x)=\tilde{u}_0(x)\,\,\hbox{ in } \,\,-\infty<x\leq h_{\infty},\eaa\right.
\end{equation*}
has a unique solution $w(t,x)\in C^{1,2}\big(\R\times (-\infty,h_{\infty}]\big)$, and 
$$w(t+s,x)- \bar{w}(t+s,x)\to 0 \,\hbox{ as } s\to\infty  \,\hbox{ in } C^{1,2}_{loc}\big(\R\times (-\infty,h_{\infty}]\big),$$
where $\bar{w}$ is the unique positive time $\omega$-periodic solution for problem
\begin{equation*}
\left\{\baa{l}
\partial_t\bar{w}-d\partial_{xx}\bar{w}=f(t,x,\bar{w})\,\hbox{ in } t>0,\, -\infty<x<h_{\infty},\vspace{3pt}\\
\bar{w}(t,h_{\infty})=0 \,\,\hbox{ for  all } \,\,t\in\R.\eaa\right.
\end{equation*}  
By the parabolic maximum principle, one has $u(t+T,x)\leq w(t,x)$ for all $t>0$, $-\infty<x<h_{\infty}$, whence
$$\limsup_{n\to\infty}u(t+n\omega,x) \leq \bar{w}(t,x)  \,\hbox{ for all } t>0, \, -\infty<x\leq h_{\infty}.$$
Furthermore, by simple modifications of  the proof of Lemma~\ref{halfsta}, one sees that $\bar{w}_T(t,x)$ (extended by 0 outside its supporting set) converges to $\bar{w}(t,x)$ as $T\to\infty$ locally uniformly in $\R\times(-\infty,h_\infty]$. Therefore, 
\begin{equation}\label{uconbw}
u(t+n\omega,x) \to \bar{w}(t,x)\,\hbox{ as }\,n\to\infty\,\hbox{ locally uniformly in }\,\R\times(-\infty,h_\infty],
\end{equation}
 which in particular implies that $u(n\omega,x)$ converges to $\bar{w}(0,x)$ as $n\to\infty$ locally uniformly in $(-\infty,h_\infty]$. Since $g(n\omega)\to-\infty$ and $h(n\omega)\to h_{\infty}$ as $n\to\infty$, and since $h_{\infty}<\infty$, it follows from the proof for the continuous dependence stated in Proposition~\ref{cdepend} that 
 \begin{equation}\label{uconbw1}
u(t+n\omega,x) \to \tilde{w}(t,x)\,\hbox{ as }\,n\to\infty\,\hbox{ locally uniformly in }\,t>0,\,-\infty<x\leq \tilde{h}(t),
\end{equation}
 where $(\tilde{w},\tilde{h})$ is the solution for the following free boundary problem
\begin{equation*}\left\{\baa{ll}
\tilde{w}_t=d\tilde{w}_{xx}+f(t,x,\tilde{w}),& -\infty<x<\tilde{h}(t),\quad t>0,\vspace{3pt}\\
\tilde{w}(t,h(t))=0,\,\,\, \tilde{h}'(t)=-\mu \tilde{w}_x(t,\tilde{h}(t)),&t>0, \vspace{3pt}\\
\tilde{h}(0)=h_{\infty},\quad \tilde{w}(0,x)=\bar{w}(0,x),& -\infty< x\leq h_{\infty}.\eaa\right.
\end{equation*} 
One then obtains from \eqref{uconbw} and \eqref{uconbw1} that $\tilde{h}(t)\equiv h_{\infty}$ and $\tilde{w}\equiv \bar{w}$. This implies that $\tilde{h}'(t)=0$ for all $t>0$, and hence $\partial_x\bar{w}(t,h_{\infty})=0$, which is a contradiction with the fact that $\partial_x\bar{w}(t,h_{\infty})<0$ by Hopf lemma.
Therefore, one gets that $g_{\infty}> -\infty$.

Once $g_{\infty}> -\infty$ is obtained, similar strategies used above would further imply that 
$h_{\infty}-g_{\infty}\leq 2R^*(y_{\infty})$ and the details will not be repeated here.
 Finally, we prove that $\lim_{t\to\infty} \max_{g(t)\leq x\leq h(t)} u(t,x)=0$. 
As a matter of fact, let $\bar{u}$ be the unique solution to the following problem 
\begin{equation*}\left\{\baa{ll}
\bar{u}_t=d\bar{u}_{xx}+f(t,x,\bar{u}),& t>0,\quad g_{\infty}<x<h_{\infty} ,\vspace{3pt}\\
\bar{u}(t,g_{\infty})=0,\,\,\, \bar{u}(t,h_{\infty})=0,&t>0, \vspace{3pt}\\
\bar{u}(0,x)=\bar{u}_0(x),&g_{\infty}\leq  x\leq h_{\infty},\eaa\right.
\end{equation*} 
where $\bar{u}_0(x)=u_0(x)$ for $x\in[g_0,h_0]$ and $\bar{u}_0(x)=0$ for $x\in [g_\infty,h_{\infty}]\setminus [g_0,h_0]$. 
It follows from the parabolic maximum principle that $0\leq u(t,x)\leq \bar{u}(t,x)$ for $t>0,\, x\in [g(t),h(t)]$. 
Furthermore, since $h_{\infty}-g_{\infty}\leq 2R^*(y_{\infty})$, the principal eigenvalue
 $\lambda_{1,(h_{\infty}-g_{\infty})/2}^{y_\infty}\geq 0$, and hence, $\lim_{t\to\infty} \bar{u}(t,x)=0$ 
uniformly in $x\in [g_{\infty},h_{\infty}]$ (see, e.g., \cite[Theorem 28.1]{he}). Therefore, 
 $\lim_{t\to\infty} \max_{g(t)\leq x\leq h(t)} u(t,x)=0$. The proof of Lemma~\ref{vanishing} is now complete.
\end{proof}

\begin{lem}\label{vspreading}
If $(g_\infty,h_\infty)=\R$, then 
$$\lim_{t\to\infty}\big|u(t,x)-p(t,x) \big|=0 \,\,\hbox{ locally uniformly in }\, x\in\R,$$
where $p(t,x)$ is the unique positive solution of problem \eqref{psteady11}.
\end{lem}
\begin{proof}
The proof follows from similar arguments as those used in the proof of \cite[Theorem 3.4]{dgp}, so we omit the details.
\end{proof}

Theorem~\ref{spva1} clearly follows directly from Lemmas~\ref{vanishing}~and~\ref{vspreading}.

 \begin{lem}\label{h0hinfty}
If $h_0-g_0\geq  2R^*(y_0)$ with $y_0=(h_0+g_0)/2$, then $(g_{\infty},h_{\infty})=\R$ and spreading always occurs.
\end{lem}
\begin{proof} We first consider the case $h_0-g_0>  2R^*(y_0)$.
Assume by contradiction that $(g_{\infty},h_{\infty}) \subsetneqq \R$. It then follows from Lemma~\ref{vanishing} that 
both $h_{\infty}$ and $g_{\infty}$ are finite, and that $\lim_{t\to\infty} \max_{g(t)\leq x\leq h(t)} u(t,x)=0$. 
On the other hand, let $\tilde{u}(t,x)$ be the unique solution of the following problem 
\begin{equation*}
\left\{\baa{l}
\partial_t\tilde{u}-d\partial_{xx}\tilde{u}=f(t,x,\tilde{u})\,\hbox{ in }\, t>0,\,\,g_0<x<h_0,\vspace{3pt}\\
\tilde{u}(t,g_0)=\tilde{u}(t,h_0)=0 \,\,\hbox{ for  all } \,\,t>0,\vspace{3pt}\\
\tilde{u}(0,x)=u_0(x)\,\,\hbox{ in } \,\,g_0\leq x\leq h_0.\eaa\right.
\end{equation*}
Since $h_0-g_0> 2R^*(y_0)$, it follows from Lemma~\ref{localsta} that 
$\lim_{t\to\infty}\tilde{u}(t,x)>0$ for all  $x\in (g_0,h_0)$.
By the parabolic maximum principle, one has $u(t,x)\geq \tilde{u}(t,x)$ for all $t>0$, $g_0\leq x\leq h_0$.
One then obtains  $\liminf_{t\to\infty}u(t,x)>0$ for all $g_0<x<h_0$, which is a contradiction. Therefore, $(g_{\infty},h_{\infty})=\R$ and spreading always occurs.

Next we consider the remaining case $h_0-g_0=  2R^*(y_0)$. 
Let $(u,g,h)$ be the unique solution of \eqref{eqf} with initial data $(u_0, g_0, h_0)$. 
Then $h(1)>h_0>g_0>g(1)$. Therefore there exist
 $\tilde g_0\in (g(1), g_0)$ and $\tilde h_0\in (h_0, h(1))$ such that $y_0$ is 
the center of the interval $[\tilde g_0, \tilde h_0]$. We now choose $\tilde u_0(x)$ 
such that it is continuous in $[\tilde g_0, \tilde h_0]$, 
\[
\tilde u_0(\tilde g_0)=\tilde u_0(\tilde h_0)=0,\; 0<\tilde u_0(x)<u(1,x) \mbox{ for } x\in (\tilde g_0, \tilde h_0).
\]
Let $(\tilde u, \tilde g,\tilde h)$ be the unique solution of \eqref{eqf} with initial data $(\tilde u_0, \tilde g_0, \tilde h_0)$. Then by the comparison principle we have
\[
h(1+t)\geq \tilde h(t),\; g(1+t)\leq \tilde g(t),\; u(1+t, x)\geq \tilde u(t, x) \mbox{ for } t>0,\; x\in [\tilde g(t),\tilde h(t)].
\]
Since $\tilde h(0)-\tilde g(0)>2R^*(y_0)$, by what has been proved above, we have $\lim_{t\to\infty}[-\tilde g(t)]=\lim_{t\to\infty} \tilde h(t)=\infty$. It follows that $h_\infty=\infty$, $g_\infty=-\infty$, and hence spreading occurs.
\end{proof}

Lemma~\ref{h0hinfty} gives the first statement of Theorem~\ref{spva2}. Next, we turn to describe the strategy for the proof of the second one. As a matter of fact, by minor modifications of the proof for \cite[Lemma 2.8]{dg1} and \cite[Lemma 3.10]{dgp}, one concludes the following two properties.

\begin{lem}\label{criteria1}
Suppose that $h_0-g_0< 2R^*(y_0)$ with $y_0=(h_0+g_0)/2$. Then there exists $\mu^0>0$ depending on $u_0$ such that spreading occurs if $\mu\geq \mu^0$. 
\end{lem}

\begin{lem}\label{criteria2}
Suppose that $h_0-g_0<2R^*(y_0)$ with $y_0=(h_0+g_0)/2$. Then there exists $\mu_0>0$ depending on $u_0$ such that vanishing occurs if $\mu\leq \mu_0$. 
\end{lem}

Based on the above two lemmas, the proof for part (ii) of Theorem~\ref{spva2} follows exactly the same arguments as those used in the proof of \cite[Theorem 2.10]{dg1}.

\end{document}